\documentclass[11pt]{amsart}
\usepackage{graphicx}
\usepackage{amssymb}
\usepackage{epstopdf} %needed for pdftex

\usepackage{ifthen}
\usepackage{chemarrow}
\usepackage{amscd}
\usepackage{amsmath}
\usepackage{xcolor}
\usepackage{hyperref} %hyperref is incompatible (now ok) with arXiv
\usepackage[all]{xy}
\xyoption{matrix}
\xyoption{arrow}

\newtheorem{thm}{Theorem}[subsection]
\newtheorem{lem}[thm]{Lemma}
\newtheorem{cor}[thm]{Corollary}
\newtheorem{prop}[thm]{Proposition}
\newtheorem{conj}[thm]{Conjecture}

\theoremstyle{definition}
\newtheorem{defn}[thm]{Definition}
\newtheorem{eg}[thm]{Example}

\theoremstyle{remark}
\newtheorem{rem}[thm]{Remark}

\numberwithin{equation}{section}

\def\vs#1{\vskip .#1 cm} %enter amount of skip wanted at #1

\def\xrarrow{\xrightarrow} %right arrow {label on top}

\def\then{\Rightarrow}

\def\into{\hookrightarrow}%or \rightarrowtail

\def\delete{\backslash}

\def\<{\left<}
\def\>{\right>}

\DeclareMathOperator{\colim}{colim}%

\DeclareMathOperator{\Emb}{Emb}
\DeclareMathOperator{\Map}{Map}

\newcommand{\sd}{SD}

\newcommand{\sdgo}[2]{\sd^{G/O}_{#1,#2}}
\newcommand{\ovsdgo}[2]{\ov\sd^{G/O}_{#1,#2}}
\newcommand{\Gamsub}[2]{\Gam_{#1,#2}}
\newcommand{\normch}{\widetilde{ch}}

\newcommand{\utd}[1]{\widetilde\cS^\d_{#1}} % unstable tangential smooth structures with fixed vertical boundary
\newcommand{\std}[2]{\widetilde\cS^s_{#1,#2}} %stable tangential smooth structures, relative case
\newcommand{\stdone}[1]{\widetilde\cS^{s}_{#1}}%stable tangential smooth structures, with just B.

\newcommand{\DWW}{\text{\rm DWW}}
\newcommand{\IK}{\text{\rm IK}}
\newcommand{\BL}{\text{\rm BL}}

\newcommand{\field}[1]{\mathbb{#1}}
\newcommand{\ZZ}{\ensuremath{{\field{Z}}}}
\newcommand{\CC}{\ensuremath{{\field{C}}}}
\newcommand{\RR}{\ensuremath{{\field{R}}}}
\newcommand{\QQ}{\ensuremath{{\field{Q}}}}

\newcommand{\commentout}[1]{}

\newcommand{\vertical}{{\text{\sf v}}}

\def\vv{^\vertical\!}

\def\ll{\lambda}

\newcommand{\cS}{\ensuremath{{\mathcal{S}}}}

\def\b{\beta}
\def\g{\gamma}
\def\Gam{\Gamma}
\def\d{\partial}

\def\e{\epsilon}

\def\r{\rho}
\def\s{\sigma}
\def\Sig{\Sigma}
\def\t{\tau}
\def\th{\theta}

\def\z{\zeta}

\def\ov{\overline}

\def\ul{\underline}
\def\wt{\widetilde}

\title{Exotic smooth structures on topological fibre bundles II}
\author{Sebastian Goette}
\address{Mathematisches Institut, Universit\"at Freiburg, Eckerstr. 1, 79104 Freiburg,
Germany}
\email{sebastian.goette@math.uni-freiburg.de}

\author{Kiyoshi Igusa}
\address{Department of Mathematics, Brandeis University, Waltham, MA 02454}
\email{igusa@brandeis.edu}


\begin{document}

\begin{abstract} We use a variation of a classical construction of A. Hatcher to construct virtually all stable exotic smooth structures on compact smooth manifold bundles whose fibers have sufficiently large odd dimension (at least twice the base dimension plus 3). Using a variation of the Dwyer-Weiss-Williams smoothing theory which we explain in a separate joint paper with Bruce Williams \cite{Second}, we associate a homology class in the total space of the bundle to each exotic smooth structure and we show that the image of this class in the homology of the base is the Poincar\'e dual of the relative higher Igusa-Klein (\IK) torsion invariant. This answers the question, in the relative case, of which cohomology classes can occur as relative higher torsion classes.
\end{abstract}

\maketitle

\tableofcontents

\setcounter{page}{1}% equal to the page number
 
 %%%%%%%%%%%%%%%%%%%%%%%%%%%%%%%%%%%%
 %
 %						Main Body
 %
 %%%%%%%%%%%%%%%%%%%%%%%%%%%%%%%%%%%%
%%\newpage

%	Part C: Arc de Triomphe
%
%	today: % Wednesday, Nov 17, 2010: Starting again

 %%%%%%%%%%%%%%%%%%%%%%%%%%%%%%%%%%%
 %
 %		section 0 {Introduction and outline}
 %
 %%%%%%%%%%%%%%%%%%%%%%%%%%%%%%%%%%%

\section*{Introduction and outline} 

%Definition of \emph{Tangential smoothing.} This refers to a homeomorphism $f:M_0\to M_1$ between smooth manifolds together with a vector bundle morphism $Tf:TM_0\to TM_1$ between their tangent bundles which is compatible with the topological derivative $Ef:EM_0\to EM_1$. There is also an equivalent formulation given by a continuous family of topological manifolds $M_t, t\in [0,1]$ which have linear Euclidean bundles $VM_t$ so that $VM_0=TM_0$ and $VM_1=TM_1$.

Higher analogues of Reidemeister torsion and Ray-Singer analytic torsion were developed by J. Wagoner, J.R. Klein, M. Bismut, J. Lott, W. Dwyer, M. Weiss, E.B. Williams, W. Dorabiala, B. Badzioch, the authors of this paper and many others. (\cite{Wagoner:higher-torsion}, \cite {Klein:thesis}, \cite{IK1:Borel2}, \cite{Bismut-Lott95}, \cite{DWW}, \cite{BG2}, \cite{Goette01}, \cite{Goette03},\cite{Goette08}, \cite{I:BookOne}, \cite{BDW09}, \cite{BDKW}). 

There are three different definitions of the higher torsion due to Igusa-Klein \cite {IK1:Borel2}, \cite{I:BookOne}, Dwyer-Weiss-Williams \cite{DWW} and Bismut-Lott \cite {Bismut-Lott95}, \cite{BG2} which are now known to be related in a precise way \cite{BDKW}, \cite{Goette08}, \cite{I:Axioms0}. In this paper we use the Igusa-Klein (\IK) torsion as formulated axiomatically in \cite{I:Axioms0} (See the review of higher torsion and it basic properties in Section \ref{subsecA13}.) The results can be translated into results for the other higher torsion invariants using the formulas relating the Dwyer-Weiss-Williams (\DWW) smooth torsion, the nonequivariant Bismut-Lott (\BL) analytic torsion and the \IK-torsion. (See \cite{BDKW}, \cite{Goette08}.)

Higher Reidemeister torsion invariants are cohomology classes in the base of certain smooth manifold bundles which can sometimes be used to distinguish between different smooth structures on the same topological manifold bundle. The main purpose of this work is to determine which cohomology classes occur as higher Reidemeister torsion invariants of exotic smooth structures on the same topological manifold bundle. We also determine to what extent the higher torsion distinguishes between different smooth structures on the same bundle.

Since the higher torsion is a sequence of real cohomology classes which are ``stable'', it can only detect the torsion-free part of the group of stable smooth structures on topological bundles. Following Dwyer, Weiss and Williams we eschew classical smoothing theory by assuming that we are given a fixed linearization (vector bundle structure) on the vertical tangent microbundle of a topological manifold bundle. We also assume that there exists at least one smoothing. With these points in mind, we give a complete answer to these two questions in the relative. Also, in the process, we give an explicit construction of ``virtually all'' exotic smooth structures on smooth manifold bundles with closed fibers of sufficiently large odd dimension.

%\newpage

\subsection{Statement of results} 

Suppose that $p:M\to B$ is a smooth bundle with fiber $X$. This means that $X,M,B$ are compact smooth manifolds and $B$ is covered by open sets $U$ so that $p^{-1}(U)$ is diffeomorphic to $U\times X$. We always assume that $B,X$ and $M$ are oriented since we need to use Poincar\'e duality. For the purpose of this introduction, we also assume that $X,B$ and $M$ are closed manifolds although we also need to consider disk bundles over $M$.

Let $T\vv M$ be the \emph{vertical tangent bundle} of $M$, i.e. the kernel of the map of tangent bundles $Tp:TM\to TB$ induced by $p$. By an \emph{exotic smooth structure} on $M$ we mean another smooth bundle $M'\to B$ together with a fiberwise tangential homeomorphism $f:M\cong M'$. This in turn means that $f$ is a homeomorphism over $B$ and that $f$ is covered by an isomorphism of vector bundles $T\vv f:T\vv M\cong T\vv M'$ which is compatible with the topological derivative of the homeomorphism $f$. (See \cite{Second}, subsection 1.3.3.) 

There are two invariants that we can associate to an exotic smooth structure $M'$ on $M$. One is the higher relative \IK-torsion invariant (Section \ref{subsecA13})
\[
	\t^\IK(M',M)\in \bigoplus_{k>0} H^{4k}(B;\RR)
\]
and the other is the \emph{relative smooth structure class} $
	\Theta(M',M)\in H_{\ast}(M;\RR)
$
which is a more complete invariant given by the following theorem which is a reinterpretation of the results of \cite{DWW} as explained in \cite{Second} and \cite{WilliamsNotes06}.

\begin{thm}\cite{Second}\label{main result of GIW}
Let $\stdone{B}(M)$ be the direct limit over all linear disk bundles $D(M)$ over $M$ of the space of all exotic smooth structures on $D(M)$.
Then $\pi_0\stdone{B}(M)$ is a finitely generated abelian group and
\[
	\pi_0\stdone{B}(M)\otimes\RR \cong \bigoplus_{k>0} H_{\dim B-4k}(M;\RR)
\]
\end{thm} 

In particular, any exotic smooth structure $M'\to B$ on $M$ gives an element of $\stdone{B}(M)$ and the corresponding element in the homology of $M$ is called the \emph{smooth structure class} of $M'$ relative to $M$ and will be denoted
\[
	\Theta(M',M)\in H_{q-4\bullet}(M):= \bigoplus_{k>0} H_{\dim B-4k}(M;\RR)
\]
We also use the indicated shortcut. The spot $\bullet$ will denote direct sum over $k>0$ as indicated. Coefficients will be in $\RR$ unless otherwise stated. We also aways denote the dimension of $B$ by $q$. 

The first main theorem of this paper is the following formula relating these invariants.

\begin{thm}[Theorem \ref{second main theorem}, Corollary \ref{second main theorem: even case}]\label{intro: key result}
\[
	D\t^\IK(M',M)=p_\ast\Theta(M',M)\in H_{q-4\bullet}(B)
\]
where $D:H^{4\bullet}(B)\cong H_{q-4\bullet}(B)$ is Poincar\'e duality and\[
	p_\ast:H_{q-4\bullet}(M)\to H_{q-4\bullet}(B)
\]
is the map in homology induced by $p:M\to B$. 
\end{thm}

This theorem can be interpreted to mean that, up to elements of finite order, differences in stable smooth structure which are not detected by higher torsion invariants are classified by homology classes in the kernel of the mapping $p_\ast$.

Combining these theorems, we obtain the answer in the stable case to the question which motivated this project, namely which cohomology classes occur as higher torsion invariants: The union of all higher \IK-torsion invariants of exotic smooth structures on all linear disk bundles $D(M)$ over $M$ span the Poincar\'e dual of the image of $p_\ast$. The answer in the unstable case is given by the following result which is a reformulation of Corollary \ref{third main theorem}.

\begin{thm}
Let $p:M\to B$ be a smooth manifold bundle whose base $B$, fiber $X$ and total space $M$ are closed oriented smooth manifold. Suppose that $\dim X$ is odd and at least $2\dim B+3$. Let $\b\in H^{4\bullet}(B)$ be a real cohomology class whose Poincar\'e dual is the image of an integral homology class in $M$. Then there exists another smooth bundle $p':M'\to B$ which is fiberwise tangentially homeomorphic to $p$ so that the relative torsion $\t^\IK(M',M)$ is a nonzero multiple of $\b$.
\end{thm}

The construction which produces these exotic smooth structures is a variation of the classical construction of Hatcher. We call our version the ``Arc de Triomphe'' (AdT) construction due to its appearance (Figure \ref{AdT figure}). The theorem above is therefore a consequence of Theorem \ref{intro: key result} and the following theorem.

\begin{thm}[AdT Theorem \ref {AdT lemma}]\label{intro: AdT Thm} When $\dim X\ge 2\dim B+3$ is odd, the relative smooth structure classes of the smooth bundles $M'\to B$ given by the AdT construction span the vector space $H_{q-4\bullet}(M)$.
\end{thm}

%\newpage

%-----------------------------------------------------------------------------------
%            sub section {Outline of the proofs}
%-----------------------------------------------------------------------------------

\subsection{Outline of the proofs}

{ %OUTLINE OF PROOF

The proofs of the results outlined above are interrelated and revolve around the proof of the key result Theorem \ref{intro: key result} which can be restated as follows. We consider two homology invariants for stable exotic smooth structures on smooth bundles $M\to B$. One is the Poincar\'e dual of the higher $\IK$-torsion and the other is the image of the smooth structure class $\Theta_M(M')$ in the homology of the base. Our theorem is that these invariant are equal. To prove this we note that these invariants are homomorphisms with a common domain, namely the group of isomorphism classes of stable exotic smooth structures on $M$, and common target, namely the direct sum of the real homology groups of $B$ in degrees $q-4k$ where $q=\dim B$. To prove that these homomorphisms are equal, we construct examples of exotic smooth structures for which we can calculate both invariants and show that the invariants agree. Then we show that our examples span the domain, tensored with $\RR$. By linearity, the invariants must agree everywhere.

The examples come from the theory of generalized Morse functions. We start with the main result of \cite{I:GMF} which we reformulate to say that the set of possible singular sets of fiberwise generalized Morse functions on $M$ produce a spanning set in the homology of $M$ in the correct degrees. These singular sets are examples of ``stratified subsets'' of $M$ with coefficients in $BO$. We can replace $BO$ with $G/O$ since they are rationally homotopy equivalent. Then we use the Arc de Triomphe construction which converts stratified subsets of $M$ with coefficients in $G/O$ into exotic smooth structures on $M$. Next we convert the Arc de Triomphe construction into an equivalent ``immersed Hatcher handle'' construction in order to compute its $\IK$-torsion.

The immersed Hatcher construction has the property that it is supported on an embedded disk bundle $E\subseteq M$. By the functorial properties of the smooth structure invariant $\Theta$ proved in \cite{Second}, the corresponding formula for higher torsion invariants proved in Theorem \ref{torsion of immersed Hatcher} below and the fact that the two invariants $p_\ast\circ\Theta$ and $D\circ\t^\IK$ agree on disk bundles, also proved in \cite{Second}, we conclude that they also agree on the immersed Hatcher construction. Therefore, these two invariants agree on the Arc de Triomphe construction and we only need to show that there are enough of these examples to generate a subgroup of the domain of finite index.

To prove this last statement we use Theorem \ref{main result of GIW} above which is derived from the \DWW-version of smoothing theory and proved in \cite{Second}. The latter result says that the group of isomorphism classes of stable exotic smooth structures on $M$, when tensored with $\RR$, is isomorphic to the direct sum of the real homology groups of the total space in degrees $\dim B-4k$. But these are exactly the homology groups spanned by elements coming from singular sets of fiberwise GMF's. This proves simultaneously the key result: $D\circ\t^\IK=p_\ast\circ\Theta$ and the AdT Theorem \ref{intro: AdT Thm}. The other theorems follow from these.
%END OF PROOF OUTLINE.
}

\subsubsection{Functorial properties of exotic smooth structures}

The key property that $D\circ \t^\IK=p_\ast\circ\Theta$ is known to hold for disk bundles and we use the functorial properties of these homology invariants to conclude that it holds for all smooth bundle. The functorial property of $\Theta$, as proved in \cite{Second}, is as follows (in the special case that $\d B$ is empty).

Suppose that $L$ is a compact smooth $q$-manifold with boundary where $q=\dim B$ and $\ll:L\to B$ is an immersion. Choose a lifting of $\ll$ to an embedding $\tilde\ll:L\to M$. Then, a fiberwise neighborhood of the image of $L$ in $M$ is a smooth disk bundle $\pi:E\to L$ with fiber $D^N$ where $N=\dim X$ is the dimension of the fiber of $p:M\to B$. The inclusion map $D(\tilde\ll):E\to M$ is a smooth embedding over $\ll$ in the sense that the following diagram commutes.
\[
%\xymatrixrowsep{10pt}\xymatrixcolsep{10pt}
\xymatrix{%begin xy matrix
E\ar[d]_\pi\ar[r]^{D(\tilde\ll)} &
	M\ar[d]^p\\
L \ar[r]^\ll&  B
	}%end xy matrix
\]
The following naturality statement for the smooth structure class $\Theta$ is proved in \cite{Second}, Corollary 2.4.3.

\begin{thm}[stratified deformation theorem] The following diagram commutes where the vertical arrows are induced by the embedding $D(\tilde\ll):E\to M$ and immersion $\ll:L\to B$.
\[
\xymatrix{
	\pi_0\std{L}{\d L}(E)\ar[d]^{D(\tilde\ll)_\ast}\ar[r]^(.4)\Theta
	&
	H_{q-4\bullet}(E)\ar[d]^{D(\tilde\ll)_\ast}\ar[r]^{p_\ast}_\cong & H_{q-4\bullet}(L)\ar[d]^{\ll_\ast}\\ 
	\pi_0\stdone{B}(M) \ar[r]^(.4)\Theta
	&
	H_{q-4\bullet}(M)\ar[r]^{p_\ast}& H_{q-4\bullet}(B)
	}%end xymatrix
\]
where $\std{L}{\d L}(E)$ is the space of stable exotic smooth structures on $E$ which agree with the given smooth structure of $E$ over $\d L$ and all homology groups have coefficients in $\RR$.
\end{thm}

\subsubsection{Hatcher handles and Arc de Triomphe}

%In Section \ref {ss: Hatcher handles} we construct what we call ``Hatcher handles''. 

The AdT construction and the immersed Hatcher handle construction (Section 2) are generalizations of a classical construction of Hatcher (reviewed in Section 1) which produces exotic smooth structures on linear disk bundles. 

Just as a standard $n$-handle is given by attaching a disk bundle $D^n(\xi)\oplus D^m(\eta)$ to the top $M\times 1$ of $M\times I$ along a fiberwise embedding $S^{n-1}\oplus D^m(\eta)\to M$, Hatcher handles are given by attaching thickenings of Hatcher's disk bundle to the top $M\times 1$ of the product $M\times I$ along certain attaching maps given embeddings $\tilde\ll:L\to M$ which lie over codimension $0$ immersions $\ll:L\to B$. We call this the \emph{immersed Hatcher construction}. The reason that $\ll$ needs to be an immersion is because, following the proofs of Lemmas \ref{first version of AdT Lemma} and \ref{stratified deformation lemma}, we see that $L$ is constructed from a fiberwise generalized Morse function on $M$ and each Morse critical point of $f_b:M_b\to\RR$ of even index gives an element of $L$ mapping to $b$. Thus, we cannot expect $\ll$ to be an embedding.

For fixed $n$ and $m$ there are two kinds of Hatcher handles which we call ``negative'' and ``positive'' Hatcher handles. The attaching map for the negative Hatcher handle $A^{n,m}(\xi,\eta)$ can  be deformed to be on top of the positive Hatcher handle $B^{n,m}(\xi,\eta)$ in such a way that they cancel as shown in Figure \ref{AdT figure}. We call this the \emph{Arc de Triomphe} (AdT) construction. This construction has as input data a ``stratified subset'' of $M$. This is a pair $(\Sig,\psi)$ where $\Sig$ is a smooth oriented $q$-manifold embedded in $M$ with the property that the projection $\Sig\to B$ has only fold singularities (Definition \ref{stratified set}. The mapping $\psi:\Sig\to G/O$ gives the data for positive and negative Hatcher handles to be attached along $\Sig_+,\Sig_-$ which are the closures of the subsets of $\Sig$ along which the projection $p:\Sig\to B$ is orientation preserving or reversing, respectively and these handles are cancelled using the AdT construction along the fold set $\Sig_0=\Sig_-\cap \Sig_+$. We denote by $\sdgo{B}{\d_0}(M)$ the group of deformation classes of stratified subsets $(\Sig,\psi)$ of $M$. The Arc de Triomphe construction thus gives a map
\[
	AdT:\sdgo{B}{\d_0}(M)\to \pi_0\std{B}{\d_0}(M)
\]
which we show to be additive in Proposition \ref {sd is a group}.

One of the key results (Lemma \ref{first version of AdT Lemma}) is that this map is rationally surjective, i.e., its cokernel is finite. To prove this we use the computation of the homotopy type of the space of generalized Morse functions \cite{I:GMF} which implies that there is a fiberwise generalized Morse function $f:M\to I$ whose singular set $\Sig(f)$ together with a suitable multiple $\xi^n$ its vector bundle data $\xi$ given by the second derivative of $f$ gives an element of $\sdgo{B}{\d_0}(M)$ which maps onto a spanning subset of the real homology group
\[
	H_{q-4\bullet}(M,M_{\d_1B})\cong \pi_0\std{B}{\d_0B}(M)\otimes \RR
\]
In the following diagram, this is expressed by saying that the curved mapping $(-1)^n2D\normch$ from $\ovsdgo{B}{\d_0B}(M)$ to $H_{q-4\bullet}(M,M_{\d_1B})$ maps onto a spanning subset where $D\normch$ is the map (Subsection \ref{ss312}) which sends $(\Sig,\psi)\in \ovsdgo{B}{\d_0B}(M)$ to the image in $H_{q-4\bullet}(M)$ of the Poincar\'e dual of the normalized Chern character (Def. \ref{def: normalized Chern character}) of the bundle over $\Sig$ associated to $\psi$:
\[
	\normch(\Sig,\psi)=\sum_{k>0}(-1)^k\z(2k+1)\tfrac12ch_{4k}(\psi\otimes\CC)\in H^{4\bullet}(\Sig;\RR)
\]
where $\z(s)=\sum\frac1{n^s}$ is the Riemann zeta function.
\[
\xymatrix{
\sdgo{B}{\d_0}(M)\ar[rr]_{{AdT}}
\ar@/_2pc/[rrr]_{(-1)^n2D\normch}
&&  
\pi_0\std{B}{\d_0}(M)\ar[r]_(.4){\Theta}
%\ar@/_2pc/[rr]_{\t^\IK}
&
 H_{q-4\bullet}(M,M_{\d_1B})%\ar[r]_{p_\ast} 
% &
% H_{q-4\bullet}(B,\d_1B)
 } % end xymatrix
\]
We know from \cite{Second} that $\Theta$ is an isomorphism. So, it suffices to show that $\Theta\circ AdT=(-1)^n2D\normch$, i.e., that this diagram commutes. This is the statement of Lemma \ref{Th AdT=2ch}.

Finally, we come to the stratified deformation lemma \ref{stratified deformation lemma} which is used to prove that every AdT construction can be deformed into an immersed Hatcher construction. This crucial lemma allows us to compute the higher \IK-torsion invariant and show that they agree with the other invariant $D\circ p_\ast\circ \Theta$ on the same collection of exotic smooth structures as those given by the AdT construction. The main theorems then follow as we have outlined above.

\subsubsection{Stratified subsets} The Arc de Triomphe construction uses ``stratified subsets'' of the bundle $M$ with coefficients in $G/O$. In the case when $M$ is a closed manifold, these are defined to be closed oriented $q$-submanifolds $\Sig\subseteq M$, where $q=\dim B$, so that the restriction of the projection map $p:M\to B$ to $\Sig$ has only \emph{fold singularities}. These are points at which $p$ is given, in local coordinates, by
\[
	p(x_1,\cdots,x_q)=(x_1^2,x_2,x_3,\cdots,x_q)
\]
Then $\Sig$ becomes the union of two submanifolds $\Sig_+$ and $\Sig_-$ where $\Sig_+$, resp. $\Sig_-$, is the closure of the set of all points at which $p|\Sig$ is nonsingular and orientation preserving, resp. orientation reversing and the fold set is $\Sig_0=\Sig_+\cap \Sig_-$. The coefficients are given by a continuous mapping $\Sig\to G/O$. 

A \emph{fiberwise generalized Morse function} (GMF) on $M$ is a smooth map $f:M\to\RR$ with the property that $f$ has only Morse and birth-death singularities on the fibers. This is equivalent to saying that the vertical derivative of $f$, as a section of $T\vv M$, is transverse to the zero section and that the singular set is a stratified subset of $M$. Then $\Sig_+$, resp. $\Sig_-$, is the closure of the union of all Morse critical points of $f$ of even, resp. odd, index and $\Sig_0$ is the set of degenerate critical points of $f$. The coefficient map $\Sig\to BO$ is given by the Hessian of $f$ at each critical point. (See \cite{I:GMF}.) Using the fact that $G/O$ is rationally equivalent to $BO$, we can lift a nonzero multiple of the coefficient map to $G/O$. (Take the direct sum of the corresponding vector bundle over $\Sig$ with itself several times.)

The main theorem of \cite{I:GMF} is that the space of GMFs on a single manifold $X$ has the $\dim X$ homotopy type of $Q(BO\wedge X_+)$. Thus a fiberwise GMF is equivalent to a section of a bundle with that fiber and a standard homotopy theoretic argument proved in Corollary 2.2.2 of \cite{Second} implies that the corresponding stratified sets with coefficients lifted to $G/O$ will represent (a multiple of) any element of $ H_{q-4\bullet}(M)$. The AdT construction produces the exotic smooth structure corresponding to this homology class using this stratified set.

%\newpage

\subsection{Comparison to \DWW-torsion}

Our key result (Theorem \ref{intro: key result} above) can be interpreted as saying that the relative higher \IK-torsion is equal to the relative higher \DWW-torsion if we define the latter to be the Poincar\'e dual of the image of the relative smooth structure class in the homology of the base. This proposed definition agrees with the following recent theorem of Badzioch, Dorabiala, Klein and Williams \cite{BDKW} but the two results do not imply each other, even if the definitions were known to agree, since the absolute higher torsion (\DWW\ or \IK) is not always defined. 

\begin{thm}[Badzioch, Dorabiala, Klein and Williams]\label{Thm of BDKW}
Suppose that $M\to B$ is a smooth unipotent bundle (Definition \ref{defn:unipotent}). Then, for all $k>0$, the degree $4k$ smooth \DWW-torsion invariant of $M$ is proportional to the \IK-torsion:
\[
	\t_{2k}^{\DWW}(M)=\ll_k\t_{2k}^\IK(M)\in H^{4k}(B;\RR)
\]
for some nonzero real number $\ll_k$ depending only on $k$.
\end{thm}

\begin{rem}
Dwyer, Weiss and Williams originally defined their higher smooth torsion in the case where the action of $\pi_1 B$ on the homology of the fiber $X$ is trivial. This definition was later extended by Badzioch, Dorabiala and Williams \cite{BDW09} to the unipotent case where $H_\ast(X;\QQ)$ has a filtration by $\pi_1B$-submodules so that the associated graded module has a trivial action of $\pi_1B$. In \cite{BDKW} Badzioch, Dorabiala, Klein and Williams show that this extended theory satisfies the axioms for higher torsion given in \cite{I:Axioms0}. Since these axioms characterize exotic higher torsion invariants up to a scalar multiple, the formula $\t_{2k}^\DWW=\ll_k\t_{2k}^\IK$ above holds for all smooth unipotent bundles.

The relation between Theorem \ref{Thm of BDKW} and our second main theorem \ref{intro: key result} is very roughly as follows. Given two smooth structures $M,M'$ on the same unipotent bundle $M\to B$, the two difference torsion invariants are defined and equal to the difference between the absolute torsions of $M,M'$:
\[
	\t^\DWW(M',M)=\t^\DWW(M')-\t^\DWW(M)
\]
\[
	\t^{\IK}(M',M)=\t^{\IK}(M')-\t^{\IK}(M)
\]
This is proved in \cite{I:BookOne} for the case of $\t^\IK$ and should be fairly straightforward to prove in the case of $\t^\DWW$ if the relative \DWW-torsion is defined correctly. Therefore, Theorem \ref {Thm of BDKW} implies that these difference torsions are proportional in the case when the bundles are unipotent. 

Our Theorem \ref {intro: key result} could be interpreted as giving a conceptual definition of the higher \DWW-difference torsion in terms of \DWW-smoothing theory, showing that it is equal to the higher \IK-difference torsion defined using Morse theory. We believe that our version of \DWW-difference torsion (defined as the Poincar\'e dual of the image of $\Theta(M',M)$ in $H_{4\bullet}(B)$) is equal to $\t^\DWW(M')-\t^\DWW(M)$. By our Theorem \ref {second main theorem}, this would be equivalent to showing that the proportionality constant in the theorem of Badzioch, Dorabiala, Klein and Williams is equal to 1:
\[
	\ll_k=1
\]
so that $\t^\DWW=\t^\IK$! However, we do not attempt to prove this here.
\end{rem}

When the fibers are closed even dimensional manifolds, the theorem above still holds by Corollary \ref {second main theorem: even case}. However, the relative higher torsion class $\t^\IK(M',M)$ is equal to zero in that case:
\[
	\t^\IK(M',M)=\t^\IK(M')-\t^\IK(M)=0
\]
since $\t^\IK(M)$ depends only on the vertical tangent bundle of $M$ over $B$ by \cite{I:Axioms0} and $M'$ has the same vertical tangent bundle as $M$ by definition of tangential homeomorphism. This leads to the following conjecture.

\begin{conj}[Rigidity conjecture]
The stable smooth structure class vanishes when the fiber is a closed oriented even dimensional manifold:
\[
	\Theta(M',M)=0
\]
In other words, rationally stably, there are no exotic smooth structures on manifold bundles with closed oriented even dimensional fibers.
\end{conj}

Theorem \ref{intro: key result} implies that $\Theta(M',M)$ must lie in the kernel of the map $p_\ast$ in the closed even dimensional fiber case since the higher relative torsion is zero in this case. The AdT construction shows that $M\times I$ admits exotic smooth structures if the fiber dimension of $M\times I\to B$ is sufficiently large and odd.

\subsection{Acknowledgments} 

Research for this project was supported by the DFG special programme ``Global Differential Geometry'' and the National Science Foundation. An earlier version of this work was presented at the 2006 Arbeitsgemeinshaft at Oberwolfach on ``Higher Torsion Invariants in Differential Topology and Algebraic K-Theory.'' This was a very helpful and enjoyable meeting at which Bruce Williams gave us his famous notes on smoothing theory \cite{WilliamsNotes06} which we have expanded into a joint paper \cite{Second} giving a thorough exposition of the background material used in this paper. The American Institute of Mathematics in Palo Alto helped us to finish this project by hosting a workshop on higher torsion in 2009. This was a very productive meeting for which the directors of AIM deserve a lot of credit for keeping us focused. The second author would also like to thank the organizers of the CMS meeting at Fredericton, New Brunswick in June, 2010 for the opportunity to present the first completed version of this paper. Finally, we would like to thank the referee for many helpful suggestions, both in exposition and content of this paper and the companion paper \cite{Second}. 

%\end{document}

%\part{Construction of exotic smooth structures}
%\newpage

%%%%%%%%%%%%%%%%%%%%%%%%%%
%
%                Section  {Hatcher's example}
%
%%%%%%%%%%%%%%%%%%%%%%%%%%

\section{Hatcher's example}

Hatcher's famous construction gives smooth disk bundles over $S^{4k}$ which are homeomorphic but not diffeomorphic to $S^{4k}\times D^n$. The exact statement is given below. 

% subsection

%-----------------------------------------------------------------------------------
%            sub section {basic construction} {homotopy theory}
%-----------------------------------------------------------------------------------

\subsection{Homotopy theory}\label{subsecA11}

John Klein helped us to find the lowest dimension in which this part of the construction works.

Suppose that $B$ is a compact smooth $q$-manifold with $q$ and $\d B=\d_0B\cup \d_1B$ as before. Let 
\[
	f:B/\d_0B\to G/O
\]
be a continuous map, i.e., $f$ is a continuous mapping on $B$ which sends $\d_0B$ to the basepoint of $G/O$, the fiber of $BO\to BG$. This classifies a stable vector bundle over $B$ which is trivial over $\d_0B$ and trivial over $B$ as a spherical fibration. Take $n\ge q+1$. Then $BO_n\to BO$ is $q+1$-connected and therefore this stable vector bundle is given by a unique oriented $n$-plane bundle $\xi$ over $B$ which is trivial over $\d_0B$.

We will show that the sphere bundle $S^{n-1}(\xi)\to B$ of $\xi$ is fiber homotopically trivial. Since $G/O$ is simply connected, we may assume that $q\ge2$ and thus $n\ge3$.

\begin{rem}\label{rem:ch(xi) span H4k(B,d0)} Since $G/O$ is rationally homotopy equivalent to $BO$, the Chern characters of all real vector bundles $\xi$ obtained in this way will span the vector space 
\[
	\bigoplus_{0<k\le q/4}H^{4k}(B,\d_0B;\RR).
\]
\end{rem}

Recall that $G_n$ is the topological monoid of all unpointed self-homotopy equivalences of $S^{n-1}$. Taking unreduced suspension we get a mapping $G_n\to F_n$ where $F_n\subset \Omega^nS^n$ is the union of the degree $\pm1$ components. It follows from a theorem of Haefliger \cite{Hf} that $(F_n,G_n)$ is $2n-3$ connected ($2n-3\ge n\ge q+1$). Furthermore, the components of $\Omega^nS^n$ are all homotopy equivalent and $\pi_kBG_n=\pi_{k-1}G_n\cong\pi_{k-1} F_n$ is stable and thus finite for $k<n$. (This also follows from the EHP sequence.) Therefore, 
\[
[B/\d_0B,BG_n]\cong[B/\d_0B,BG]
\]
{for $n>q$. So, the composition}
\[
	B/\d_0B\xrarrow{\xi}BO_n\to BG_n
\]
is null homotopic for $n>q$. This implies that the sphere bundle $S^{n-1}(\xi)$ associated to $\xi$ is fiberwise homotopy equivalent to the trivial bundle:
\[
	g:S^{n-1}(\xi)\simeq S^{n-1}\times B
\]
and this trivialization agrees with the given trivialization over $\d_0B$.

Take the fiberwise mapping cone of $g$. This gives a fibration over $B$ whose fibers are contractible $n$-dimensional cell complexes which are homeomorphic to the standard $n$-disk over $\d_0B$. When we thicken this up we will get an exotic smooth structure on a trivial disk bundle over $B$.

\begin{rem}
For any space $X$ recall \cite{Adams, HusemollerFB} that $J(X)$ is the group of stable vector bundles over $X$ modulo the equivalence relation that $\xi\sim\eta$ if the sphere bundles over $\xi$ and $\eta$ are fiberwise homotopy equivalent. The group operation is fiberwise join which corresponds to direct sum of underlying bundles. If $\xi$ is any vector bundle over $X$ then $J(\xi)$ denotes its image in $J(X)$. If $X$ is a finite complex then it is well known that $J(X)$ is a finite group. (See, e.g. \cite{HusemollerFB}.) The above argument shows that if $J(\xi)$ is trivial in $J(B/\d_0B)$ and $\dim \xi>\dim B$ then the sphere bundle of $\xi$ is fiberwise homotopically trivial.
\end{rem}

% subsection

%\newpage
%-----------------------------------------------------------------------------------
%            sub section {thickening}
%-----------------------------------------------------------------------------------

\subsection{Thickening}\label{subsecA12}

We have a family of finite cell complexes over $B$ which we want to thicken to get a manifold bundle. If we embed this fibration in $D^N\times B$ and take a ``regular neighborhood'' we will get a smooth $N$ disk bundle over $B$ which is homeomorphic but not diffeomorphic to $D^N\times B$.

We start by thickening the trivial sphere bundle $S^{n-1}\times B$ to get $S^{n-1}\times I\times D^{m}\times B$. This is the trivial bundle over $B$ with fiber $S^{n-1}\times I \times D^{m}$. We also need this to be embedded in a trivial disk bundle $D^{n}\times D^m\times B$ in a standard way. We will take the obvious embedding
\[
	f:S^{n-1}\times I\times D^{m}\into D_2^n\times D^m
\]
given by $
	f(x,y,z)=\left(%\tfrac12
	(1+y)x,z
	\right)
$ where $D^n_2$ is the $n$-disk of radius 2. Then the closure of the complement of the image of $f$ in $D^n_2\times D^m$ is $D^n\times D^m$. Note that $S^{n-1}\times0\times D^m$ is mapped into $S^{n-1}\times D^m$, the side of the ``donut hole''. We also need a fixed orientation preserving embedding $i:D^{n+m}\into S^{n-1}\times I\times D^m$ which we call the \emph{basepoint disk}. Assuming that $n\ge2$, $i$ is unique up to isotopy.

We attach an $n$-handle 
$
	D^n(\xi)\oplus D^m(\eta)
$
to this (with $\eta$ necessarily being a complementary bundle to $\xi$)  to fill in the donut hole and create a smooth (after rounding corners) bundle over $B$ with fiber 
\[
S^{n-1}\times I\times D^{m}\cup D^n\times D^m\cong D^{n+m}
\]
The data needed to attach such a handle embedded in $D_2^n\times D^m\times B$ is a smooth embedding of pairs
\[
	D(j):(D^{n}(\xi),S^{n-1}(\xi))\oplus D^m(\eta)\to (D^{n},S^{n-1})\times D^m\times B
\]
%where $D^m$ represents the hemisphere in the boundary of $D^{m+1}$.

%
\begin{figure}[htbp]
\begin{center}
%
%\vs5
{
\setlength{\unitlength}{1cm}
%\centerline
{\mbox{
\begin{picture}(7,2.5)
\put(0,-1.5){
      \thicklines
      \put(0,2){\line(1,0){1.5}}
      \put(0,2){\line(0,1){1.5}}
      \put(0,3.5){\line(1,0){1.5}}
      \put(1.5,2){\line(0,1){1.5}}
      \put(5,2){\line(1,0){1.5}}
      \put(5,2){\line(0,1){1.5}}
     \put(5,3.5){\line(1,0){1.5}}
      \put(6.5,2){\line(0,1){1.5}}
	\thinlines
	\put(1.5,2){\line(1,0){3.5}}
	\put(1.5,3.5){\line(1,0){3.5}}
      \qbezier(1.5,2.4)(1.9,2.8)(3.25,2.4)
      \qbezier(3.25,2.4)(4.6,2)(5,2.4)
      \qbezier(1.5,3.1)(1.9,3.5)(3.25,3.1)
      \qbezier(3.25,3.1)(4.6,2.7)(5,3.1)
\thicklines
      \put(0,.35){
      \qbezier(1.5,2.4)(1.9,2.8)(3.25,2.4)
      \qbezier(3.25,2.4)(4.6,2)(5,2.4)
      }
\thinlines
      \put(5.8,2.75){\circle{.5}}
      \qbezier(6,3)(6.7,3.6)(7.2,3.5)
      \put(7.4,3.4){$i(D^{n+m})$}
      \put(3.1,3.7){$D^{n}$}
      \put(5.1,3.7){$S^{n-1}\times I$}
      \put(6.6,2.7){$D^m$}
      \put(2,1.4){$D(j)(D^n(\xi)\times 0)$}
      \put(3.4,1.8){\line(0,1){0.9}}
      \qbezier(3.4,2.7)(3.4,2.5)(3.3,2.5)
      \qbezier(3.4,2.7)(3.4,2.5)(3.5,2.5)
      }
\end{picture}}
}}
%\vs5
%\caption{default}
\end{center}
\end{figure}
This embedding $D(j)$ is essentially given by $j$ its restriction to the core $D^n(\xi)\times0$. %A scaling factor $\frac12$ is needed to fit this in the unit disk bundle $D^{n+m}\times B$.

\begin{lem}\label{embedding lemma}
If $m>n>q$ then there is a smooth fiberwise embedding of pairs:
\[
	j:(D^n(\xi),S^{n-1}(\xi))\to (D^n,S^{n-1})\times D^m\times B
\]
over $B$ which is the standard embedding over $\d_0B$ and which is transverse to $S^{n-1}\times D^m$. Furthermore, if $m\ge q+3$ then this fiberwise embedding will be unique up to fiberwise isotopy.
\end{lem}

\begin{proof} When $q=0$, this holds by transversality. So suppose $q>0$. We use \cite[Thm 6.5]{I:Stability} which says that the inclusion 
\[
	\Emb((D^n,S^{n-1}),(W^{n+m},\d_0W))\to \Map((D^n,S^{n-1}),(W^{n+m},\d_0W))
\]of the smooth embedding space into the mapping space is $c$-connected where
\[
	c=m-n-1+\min(s,n,m-2,n+m-4)
\]
and $s$ is the connectivity of the pair $(W,\d_0W)$. In our case $s=n-1$. So the condition $m>n>q>0$ implies that $c\ge q$ giving the existence part of the lemma and if $m\ge q+3$ then either $m\ge n+2$ or $n\ge q+2$ and we get $c>q$ which implies the uniqueness part.
\end{proof}

The embedding $j$ gives an $m$-dimensional normal bundle $\eta$ for $\xi$ and a smooth codimension 0 embedding
\[
	D(j):D^{n}(\xi)\oplus D^m(\eta)\to D^{n}\times D^m\times B
\]
Restricting this to $\d D^n(\xi)\oplus D^m(\eta)$ we get a fiberwise embedding
\[
	S(j):S^{n-1}(\xi)\oplus D^m(\eta)\to S^{n-1}\times D^m\times B
\]
We can use $S(j)$ to construct a smooth bundle (with corners rounded):
\[
	E^{n,m}(\xi)=D^n(\xi)\oplus D^m(\eta)\cup_{S(j)} S^{n-1}\times I\times D^{m}\times B.
\]
We can also use $D(j)$ to embed this in the trivial disk bundle of the same dimension:
\[
	F(j)= D(j)\cup f_B:E^{n,m}(\xi)\into D_2^n\times D^m\times B
\]
where $f_B=f\times id_B$. This is {\bf Hatcher's example}. Since $m>q$, the $m$-plane bundle $\eta$ is the stable complement to $\xi$ and is thus uniquely determined. If $m\ge q+3$ then, up to fiberwise diffeomorphism, $E(\xi)$ is independent of the choice of $j$. Finally, we note the crucial point that the bundle $E(\xi)$ is canonically diffeomorphic to the trivial bundle over $\d_0B$.

%\newpage

\subsection{Higher Reidemeister torsion}\label{subsecA13}

We will briefly review the definition and basic properties of higher Reidemeister torsion invariants following \cite{I:Axioms0}, in particular the handlebody formula Theorem \ref{higher torsion of fiberwise handlebody}. Then we will use these formulas to calculate the higher \IK-torsion for Hatcher's example. The analytic torsion of Hatcher's example is computed in \cite{Goette03} as an application of the handlebody formula for analytic torsion proved in that paper.

\begin{defn}\label{defn:unipotent}
A smooth bundle $M\to B$ with compact oriented manifold fiber $X$ and connected base $B$ is called \emph{unipotent} if the rational homology of the fiber $H_\ast(X;\QQ)$, considered as a $\pi_1B$-module has a filtration by submodules so that the action of $\pi_1B$ on the subquotients is trivial. For example, any oriented sphere bundle is unipotent. 
\end{defn}

\begin{defn}\label{defn:axiomatic higher torsion}
A {\bf higher torsion invariant} is a real characteristic class $\t(M)\in H^{4\bullet}(B;\RR)$ of smooth unipotent bundles $M\to B$ with closed fibers satisfying the following two axioms.

(\emph{Additivity}) Suppose that $M=M_0\cup M_1$ where $M_0,M_1$ are unipotent compact manifold subbundles of $M$ which meet along their fiberwise boundary $M_0\cap M_1=\d\vv M_0= \d\vv M_1$. Then
\[
	\t(M)=\frac12\t(DM_0)+\frac12\t(DM_1)
\]
where $DM_i$ is the \emph{fiberwise double} of $M_i$ (the union of two copies of $M_0$ along their fiberwise boundary).

(\emph{Transfer}) Suppose that $M\to B$ is a unipotent bundle with closed fibers and $S^n(\xi)\to M$ is the sphere bundle of an $SO(n+1)$-bundle $\xi$ over $M$. Then $S^n(\xi)$ is a unipotent bundle over both $M$ and $B$ and thus has two higher torsion invariants $\t_M,\t_B$. These are required to be related as follows.
\[
	\t_B(S^n(\xi))=\chi(S^n)\t(M)+tr^M_B(\t_M(S^n(\xi)))
\]
where $\chi$ is Euler characteristic and $tr^M_B:H^\ast(M)\to H^\ast(B)$ is the \emph{transfer}. (See \cite{I:Axioms0} for more details.)
\end{defn}

\begin{thm}\cite{I:Axioms0}\label{even torsion is a tangential invariant}
If $M\to B$ has closed even dimensional fibers then any higher torsion invariant $\t(M)$ depends only on the fiberwise tangential homeomorphism type of $M$. In other words, for any exotic smooth structure $M'$ for $M$, we have $\t(M')=\t(M)$.
\end{thm}

\begin{rem}\label{rem:extension to not closed fibers}
Any higher torsion invariant can be extended to unipotent bundles $M\to B$ with compact oriented fibers using the following formula.
\[
	\t(M):=\frac12\t(DM)+\frac12\t(\d\vv M)
\]
The sign is positive ($+$) since the double $DM$ has only one copy to $\d\vv M$.
\end{rem}

We say that a higher torsion invariant $\t$ is \emph{stable} (``{exotic}'' in \cite{I:Axioms0}) if
\[
	\t(M)=\t(D(\xi))
\]
for any oriented linear disk bundle $D(\xi)$ over $M$ considered as a unipotent bundle over $B$.

\begin{thm}\cite{I:Axioms0}\label{uniqueness of exotic torsion}
The higher \IK-torsion $\t^\IK$ is a stable higher torsion invariant. Conversely, any stable higher torsion invariant is proportional to $\t^\IK$ with a possibly different proportionality constant in each degree $4k$ (and is zero in other degrees).
\end{thm}

\begin{thm}[Badzioch, Dorabiala, Klein and Williams]\label{Lem of BDKW}
The \DWW-higher smooth torsion is a stable higher torsion invariant. Consequently, it is proportional to \IK-higher torsion in every degree.
\end{thm}

\begin{rem}
Analytic torsion does not satisfy the stability condition and is therefore not proportional to \IK-torsion or \DWW-smooth torsion. See \cite{Goette08} for a precise formula relating Bismut-Lott analytic torsion to \IK-higher torsion.
\end{rem}

In \cite{I:Axioms0} it is shown that the three properties: Additivity, Transfer and Stability imply that the higher torsion of any linear sphere bundle is proportional to the Chern character. For \IK-torsion the proportionality constant is given by the following definition.

\begin{defn}\label{def: normalized Chern character}
If $\xi$ is a real vector bundle over $B$, we define the {\bf normalized Chern character} of $\xi$ to be the real cohomology class $\normch(\xi)=\sum_{k>0}\normch_{4k}(\xi)$ where
\[
	\normch_{4k}(\xi)=(-1)^k\z(2k+1)\tfrac12ch_{4k}(\xi\otimes\CC)\in H^{4k}(B;\RR)
\]
and $\z$ is the Riemann zeta function.
\end{defn}

\begin{thm}\cite{I:ComplexTorsion}
For any linear oriented sphere bundle $S^n(\xi)\to B$, we have
\[
	\t^\IK(S^n(\xi))=(-1)^{n}\normch(\xi)
\]
\end{thm}

To calculate the higher torsion of Hatcher's example, we use the following   formula. Suppose that a smooth bundle $M\to B$ has a fiberwise handlebody decomposition:
\[
	M=\bigcup D(\xi_i)\oplus D(\eta_i)
\]
where $\xi_i,\eta_i$ are oriented vector bundles over $B$ of dimension $n_i,m_i$ and $D^{n_i}(\xi_i)\oplus D^{m_i}(\eta_i)$ is attached to lower handles along $S^{n_i-1}(\xi_i)\oplus D^{m_i}(\eta_i)$. In other words, $D^{n_i}(\xi_i)\oplus D^{m_i}(\eta_i)$ is an $n_i$-handle with \emph{core} $D^{n_i}(\xi_i)$. 

\begin{thm}\label{higher torsion of fiberwise handlebody}
The higher \IK-torsion of the fiberwise handlebody $M$ is given by
\[
	\t^\IK(M)=\sum (-1)^{n_i}\normch(\xi_i)
\]
\end{thm}

\begin{rem}\label{rem:handlebody lemma}
This theorem is proved in \cite{I:Axioms0}, Lemma 6.6, inductively on the number of handles using the relative additivity property:
\[
	\t^\IK(M\cup D^n(\xi)\oplus D^m(\eta),M)=(-1)^n\normch(\xi)
\]
when the $n$-handle $D^n(\xi)\oplus D^m(\eta)$ is attached to $\d\vv M$ along $S^{n-1}(\xi)\oplus D^{m}(\eta)$ by any fiberwise embedding. We will also use this relative formula. Note that the $n$-handle is actually th pair $(D^n(\xi)\oplus D^m(\eta),S^{n-1}(\xi)\oplus D^{m}(\eta))$. But we refer to $D^n(\xi)\oplus D^m(\eta)$ as the $n$-handle with \emph{base} $S^{n-1}(\xi)\oplus D^{m}(\eta)$.
\end{rem}

%Summarizing the construction above and the easy calculation of the higher torsion of this bundle we get the following well known theorem.

The following theorem summarizes Hatcher's construction and gives its two main properties proved below.

\begin{thm}\label{thm: Hatcher's example}
 Suppose that $B$ is a smooth $q$-manifold and $m>n>q$. Suppose that $\xi$ is an $n$-plane bundle over $B$ which is trivial over $\d_0B\subset\d B$ so that $J(\xi)=0\in J(B/\d_0B)$. Then Hatcher's construction gives a smooth bundle $E^{n,m}(\xi)$ over $B$ with fiber $D^{n+m}$. Furthermore:
 \begin{enumerate}
 \item This bundle is fiberwise diffeomorphic to the trivial bundle over $\d_0B$ and {fiberwise homeomorphic to the trivial bundle over $B$ with fiber $D^{n+m}$.} 
 \item The higher \IK-torsion is given by
\[
	\t^\IK(E^{n,m}(\xi))=(-1)^n\normch(\xi)
\]
\end{enumerate}
\end{thm}

\begin{proof}
The higher torsion calculation follows from Theorem \ref{higher torsion of fiberwise handlebody} since $E^{n,m}(\xi)$ is given by attaching the $n$-handle $D^n(\xi)\oplus D^m(\eta)$ to a trivial bundle. 

%using the Framing Principle from \cite{I:BookOne,I:ComplexTorsion}. Here we use the version in \cite{I:Axioms0} which says that, given a smooth handlebody structure on the fibers with handles attached in the same order for each fiber, the axiomatic higher torsion is defined and equal to a linear combination of the suitably normalized Chern characters of the bundles giving the core and cocores of the handle. For \IK-torsion the coefficient for the cocore is zero and the coefficient for the core is $(-1)^k\z(2k+1)$ which is what we are using.

The bundle is topologically trivial by the Alexander trick. (The topological group of homeomorphism of the disk $D^{n+m}$ which are the identity on the southern hemisphere is contractible.)
\end{proof}

Take $q=4k,n=4k+1, m\ge 4k+2, B=S^{4k}$ and using the well known fact that the order of the image of the $J$-homomorphism $J:\pi_{4k-1}O\to \pi_{4k-1}^s$, which we denote $a_k$, is the denominator of $B_k/4k$ where $B_k$ is the $k$-th Bernoulli number \cite{Adams}, we get the following.

\begin{cor}
For any $k>0, N\ge 8k+3$ Hatcher's construction gives a smooth $N$-disk bundle over $S^{4k}$ which is tangentially homeomorphic to $D^N\times S^{4k}$ but has higher \IK-torsion invariant $\tau^\IK_{2k}\in H^{4k}(S^{4k};\RR)$ equal to $\zeta(2k+1)a_{k}$ times the generator of $H^{4k}(S^{4k};\ZZ)$ for $k$ odd and half of that number when $k$ is even.  In both cases this gives a nontrivial element of $\pi_{4k-1} Dif\!f(D^N)/O_N\otimes\RR$.
\end{cor}

\begin{proof}
It follows from Bott periodicity (\cite{Bott59}, \cite[18.9]{HusemollerFB}) that the Chern character of the stable complex vector bundle over $S^{2k}$ corresponding to a generator of $\pi_{2k}BU=\ZZ$ is equal to a generator of $H^{2k}(S^{2k};\ZZ)$. Also, the homotopy fiber sequence $BO\to BU\to \Omega^6BO$ given by the inclusion map $O\to U$ implies that the generator of $\pi_{4k}BO$ maps to the generator of $\pi_{4k}BU$ for $k$ even and to twice the generator when $k$ is odd. The generator of the kernel of the $J$-homomorphism is $a_k$ times this element. By the theorem above, the higher torsion of this exotic bundle is given by multiplying this element by $\frac12\z(2k+1)$ giving the formula in the corollary up to sign. We can make the sign positive by taking the other generator of the kernel of the $J$-homomorphism in Hatcher's construction.
\end{proof}

% Section:

%\newpage
%%%%%%%%%%%%%%%%%%%%%%%%%%
%
%                Section  {Variations of Hatcher's construction}
%
%%%%%%%%%%%%%%%%%%%%%%%%%%

\section{Variations of Hatcher's construction} We need several variations and extensions of Hatcher's construction in order to construct a full rank subgroup of the group of all possible tangential smooth structures on a smooth manifold bundle with sufficiently large odd dimensional fibers. The idea is to construct ``positive'' and ``negative'' ``suspensions'' of Hatcher's basic construction which will cancel. We call this the ``Arc de Triomphe'' construction due to the appearance of the figures used to explain the construction. Since the stabilization of bundles with even dimensional fibers includes bundles whose fiber dimensions are arbitrarily large and odd, this construction also produces ``all'' stable tangential smooth structures on bundles with even dimensional fibers.

%-----------------------------------------------------------------------------------
%            subsection {Arc de Triomphe: basic construction}
%-----------------------------------------------------------------------------------

\subsection{Arc de Triomphe: basic construction}\label{ss:AdT}\label{subsecA21}

%If we use the embedding $F(j):E^{n,m}\into D^{n+m}\times B$ we can construct the ``suspension'' of $E^{n,m}(\xi)$. This will be equivalent to the Hatcher disk bundle $E^{n+1,m}(\xi)$. By a construction which we call ``Arc de Triomphe'' we will see that a suitably defined union of this

There are two ``suspensions'' of $E^{n,m}$ to one higher dimension. We will see that their union is trivial:
\[
	E^{n,m+1}(\xi)\cup E^{n+1,m}(\xi)\cong D^{n+m+1}\times B
\]
This is in keeping with the calculation of their higher torsions:
\[
	\t^\IK(E^{n,m+1}(\xi))+\t^\IK(E^{n+1,m}(\xi))=(-1)^n\normch(\xi)+(-1)^{n+1}\normch(\xi)=0
\]
and the handlebody theorem \ref{higher torsion of fiberwise handlebody} which implies that the higher torsion of a union of fiberwise handlebodies is the sum of torsions of the pieces.

The {\bf positive suspension} of $E^{n,m}(\xi)$ is defined simply as the product (with corners rounded):
\[
	\s_+E^{n,m}(\xi)=E^{n,m}(\xi)\times I
\]
An examination of the definitions shows that this is the same as $E^{n,m+1}(\xi)$.

The {\bf negative suspension} of $E^{n,m}(\xi)$ uses the embedding $F(j)=D(j)\cup f_B:E^{n,m}(\xi)\into D_2^n\times D^m\times B$ and is defined as follows.
\[
	\s_-E^{n,m}(\xi)=D_2^n\times D^m\times[-1,0]\times B\cup_{F(j)\times 0} E^{n,m}(\xi)\times I\cup_{F(j)\times 1} D_2^n\times D^m\times [1,2]\times B
\]
This is a subbundle of $D_2^n\times D^m\times[-1,2]\times B$. We claim that $\s_-E^{n,m}(\xi)$ is a model for $E^{n+1,m}(\xi)$ over $B$ in the sense that the construction of $E^{n+1,m}(\xi)$, which may not be unique, could give $\s_-E^{n,m}(\xi)$. (We view $\xi$ as a stable vector bundle.) Lemma \ref{embedding lemma} then tells us that we have uniqueness after stabilizing just once:
\[
	\s_-E^{n,m}(\xi)\times I\cong E^{n+1,m}(\xi)\times I=E^{n+1,m+1}(\xi)
\]
since $m+1\ge q+3$. To verify this claim note that $\s_-E^{n,m}(\xi)$ contains the trivial bundle over $B$ with fiber
\[
	F=D^n\times D^m\times[-1,0]\cup
	S^{n-1}\times I\times D^m\times [0,1]
	\cup D^n\times D^m\times[1,2]
\]
which is diffeomorphic to $S^n\times D^{m+1}$ after its corners are rounded. On this is attached the $n+1$ handle $D^n(\xi)\oplus D^m(\eta)\times I$ which is equivalent to $D^{n+1}(\xi)\oplus D^m(\eta)$ after corners are rounded. Since $D^{n+1}(\xi)$ is the core of this handle, the result is $E^{n+1,m}(\xi)$.

When we take the union of the positive and negative suspensions of $E^{n,m}(\xi)$, they cancel. This will follow from the following lemma which does not require proof.

\begin{lem}\label{first trivial lemma}
Suppose that $E_0,E_1$ are compact smooth manifold bundles over $B$ with the same fiber dimension. Let $f:E_0\to E_1$ be a smooth embedding over $B$. Then
\[
	E_0\times [0,1]\cup_{f\times 1}E_1\times [1,2]
\]
is fiberwise diffeomorphic to $E_1\times I$ after rounding off corners.
\end{lem}

\begin{rem}\label{basic AdT} The example that we have in mind is
\[
	E^{n,m}(\xi)\times [0,1]\cup_{F(j)\times 1} D^n\times D^m\times[1,2]\times B\cong D^n\times D^m\times I\times B
\]
We denote the construction on the left by $V^{n,m}(\xi)$.
\end{rem}

Next we use another trivial lemma:

\begin{lem}\label{second trivial lemma}
Suppose that $\d\vv E_1= \d_0E_1\cup \d_1E_1$ where $ \d_iE_1$ are smooth manifold bundles over $B$ with the same fiberwise boundary. Let $f,g:\d_0E_1\to\d\vv E_0$ be smooth embeddings over $B$ which are fiberwise isotopic. Then $
	E_0\cup_f E_1
$ and $
	E_0\cup_g E_1
$ are fiberwise diffeomorphic over $B$ after rounding off the corners.
\end{lem}

In our example, $\d_0E_1$ will be a disk bundle. So, we need the following well-known lemma.

{
\begin{lem}\label{third trivial lemma}
Suppose that $D,D_0$ are smooth $n$-disk bundles over $B$ so that $D_0$ is a subbundle of $D$ which is disjoint from the fiberwise boundary $\d\vv D$. Let $E\to B$ be another smooth manifold bundle with fiber $F$. Then any two fiberwise embeddings $D\to E$ over $B$ are fiberwise isotopic if and only if their restrictions to $D_0$ are fiberwise isotopic.
\end{lem}

\begin{proof} Necessity of the condition is clear. To prove sufficiency, it suffices to show that there is an isotopy of $D$ into $D_0$, i.e., a smooth family of embeddings $f_t:D\to D$ over $B$ so that $f_0$ is the identity and $f_1(D)\subseteq D_0$. Then, for any two embeddings $g,h:D\to E$ whose restrictions to $D_0$ are fiberwise isotopic, we can first compose these embeddings with the isotopy $f_t$ then use the given isotopy $g|D_0\circ f_1\simeq h|D_0\circ f_1$.

To construct the isotopy $f_t$ we triangulate the base and construct the isotopy over the simplices one at a time and use the isotopy extension theorem. For each $q\ge -1$ we will construct an embedding $f_q:D\to D$ over $B$ with the following two properties.
\begin{enumerate}
\item $f_q$ is fiberwise isotopic to the identity map $id_D$.% fixing $\d_0B$.
\item $f_q(D)$ is contained in $ int\,D_0$, the fiberwise interior of $D_0$ over $B^q$, the $q$-skeleton of $B$ under the triangulation.
\end{enumerate}
Start with $q=-1$ when $B^{-1}=\emptyset$. Then $f_{-1}=id_D$ satisfies all conditions. Now suppose that $q\ge0$ and $f_{q-1}$ has been constructed. Then on each $q$-simplex $\Delta^q$ of $B$, the bundle pair $(D,D_0)$ is trivial. So we may assume they are product bundles
\[
	(D,D_0)|\Delta^q=(D_2^n\times \Delta^q,D^n\times \Delta^q)
\]
where $D_2^n$ is the disk of radius $2$ in $\RR^n$. We are given that $f_{q-1}$ sends $D$ into $int\,D_0$ over $\d\Delta^q$. Since $f(D)\subset int\,D_0$ is an open condition, this also holds over a neighborhood of $\d\Delta^q$. Then there is no problem finding an isotopy of $f_{q-1}$ to some $f_q$ over $B^q$ fixing a neighborhood of $B^{q-1}$ so that $f_q$ sends $D$ into $int\,D_0$ over $B^q$. By the isotopy extension theorem, this isotopy extends to an isotopy over all of $B$ completing the induction. When $q$ reaches the dimension of $B$, we are done.
\end{proof}
}

We use Lemmas \ref{second trivial lemma} and \ref{third trivial lemma} for
\[
	E_1 =E^{n,m}(\xi)\times [0,1]\cup_{F(j)\times 1} D_2^n\times D^m\times[1,2]\times B
\]
$\d_0E_1=E^{n,m}(\xi)\times 0$ and $E_0=M\times [-1,0]$ with
\[
	M=E^{n,m}(\xi)\cup_{h_B} D_2^n\times D^m\times B
\]
where $h_B=h\times id_B$ and $h$ is an orientation reversing diffeomorphism of a disk $D_0^{n+m-1}$ embedded in $\d(D_2^n\times D^m)$ onto another disk $D_1^{n+m-1}$ embedded in $S^{n-1}\times 1\times D^m$ (the outside surface of the donut). The pasting map $h$ needs to be orientation reversing in order for orientations of the two pieces to agree. Assuming that $n\ge 2$, $h$ is unique up to isotopy. And it is a special case of Lemma \ref{first trivial lemma} that $M$ is fiberwise diffeomorphic to $E^{n,m}(\xi)$.

Note that both pieces of $M$ contain the product bundle $S^{n-1}\times I\times D^m\times B$ and each of these contains a basepoint disk. We call the two embeddings $i_0, i_1:D^{n+m}\times B\to M$. Since these two embeddings have image in the product bundle
\[
	(S^{n-1}\times I\times D^m\cup_h D^n_2\times D^m)\times B\subseteq M
\]
which has connected fibers and since $i_0,i_1$ are equal to fixed orientation preserving embeddings on every fiber, the two embeddings are isotopic and therefore the two bundle embeddings $i_0,i_1$ are fiberwise isotopic.

As an example of Lemma \ref{second trivial lemma}, take the mapping $f:\d_0E_1\to\d\vv E_0$ to be the inclusion map
\[
	f:E^{n,m}(\xi)\times 0\subseteq M\times 0\subseteq\d\vv E_0
\] 
and $g:\d_0E_1\to \d\vv E_0$ to be the embedding:
\[
	g:E^{n,m}(\xi)\times 0\xrarrow{F(j)}D_2^n\times D^m\times B\subseteq M\times 0\subseteq\d\vv E_0
\]
We claim that $f$ and $g$ are fiberwise isotopic. To see this we restrict both maps to the basepoint disk $i(D^{n+m})\times B\subseteq S^{n-1}\times I\times D^m\times B\subseteq E^{n,m}(\xi)\times 0$. The restriction of $f$ to the basepoint disk is $i_0$ and the restriction of $g$ to the basepoint disk is $i_1$. We have just seen that $i_0$ and $i_1$ are fiberwise isotopic. Therefore, by Lemma \ref{third trivial lemma}, $f$ and $g$ are fiberwise isotopic. Therefore, by Lemma \ref{second trivial lemma},
\[
	M\times [-1,0]\cup_f E_1\cong M\times [-1,0]\cup_g E_1
\]
where $\cong$ indicates fiberwise diffeomorphism over $B$. But, when we attach $E_1$ on top of $D^n\times D^m\times B\times [-1,0]$ using the map $F(j)$ we get exactly the negative suspension $\s_-E^{n,m}(\xi)$. So, we have a diffeomorphism which preserves all the corner sets:
\[
	M\times [-1,0]\cup_g E_1= \s_-E^{n,m}(\xi)\cup_{h_B}\s_+E^{n,m}(\xi)
\]
and
\[
	M\times [-1,0]\cup_f E_1= V^{n,m}(\xi)\cup_{h_B}D_2^n\times D^m\times B\times [-1,0]\cong D^{n+m+1}\times B
\]
where $V^{n,m}(\xi)$ is given in Remark \ref {basic AdT}.
Since $h$ is unique up to isotopy, any two choices of $h$ will produce fiberwise diffeomorphic bundles. So we get the following. (See Figure \ref{AdT figure}. The notation $E_1=A^{n,m}(\xi,\eta)$ is from subsection \ref {ss: Hatcher handles}.)

\begin{prop}[basic cancellation lemma]\label{basic AdT cancellation lemma}
The oriented union of the positive and negative suspensions of $E^{n,m}(\xi)$ glued together along fixed $n+m$ disk bundles in the fixed parts of their boundary is fiberwise diffeomorphic to the trivial $n+m+1$ disk bundle over $B$:
\[
	 \s_-E^{n,m}(\xi)\cup_{h_B}\s_+E^{n,m}(\xi)\cong D^{n+m+1}\times B.
\]
\end{prop}
\begin{center}
\begin{figure}\label{AdT figure}
       \includegraphics[width=5.5in]{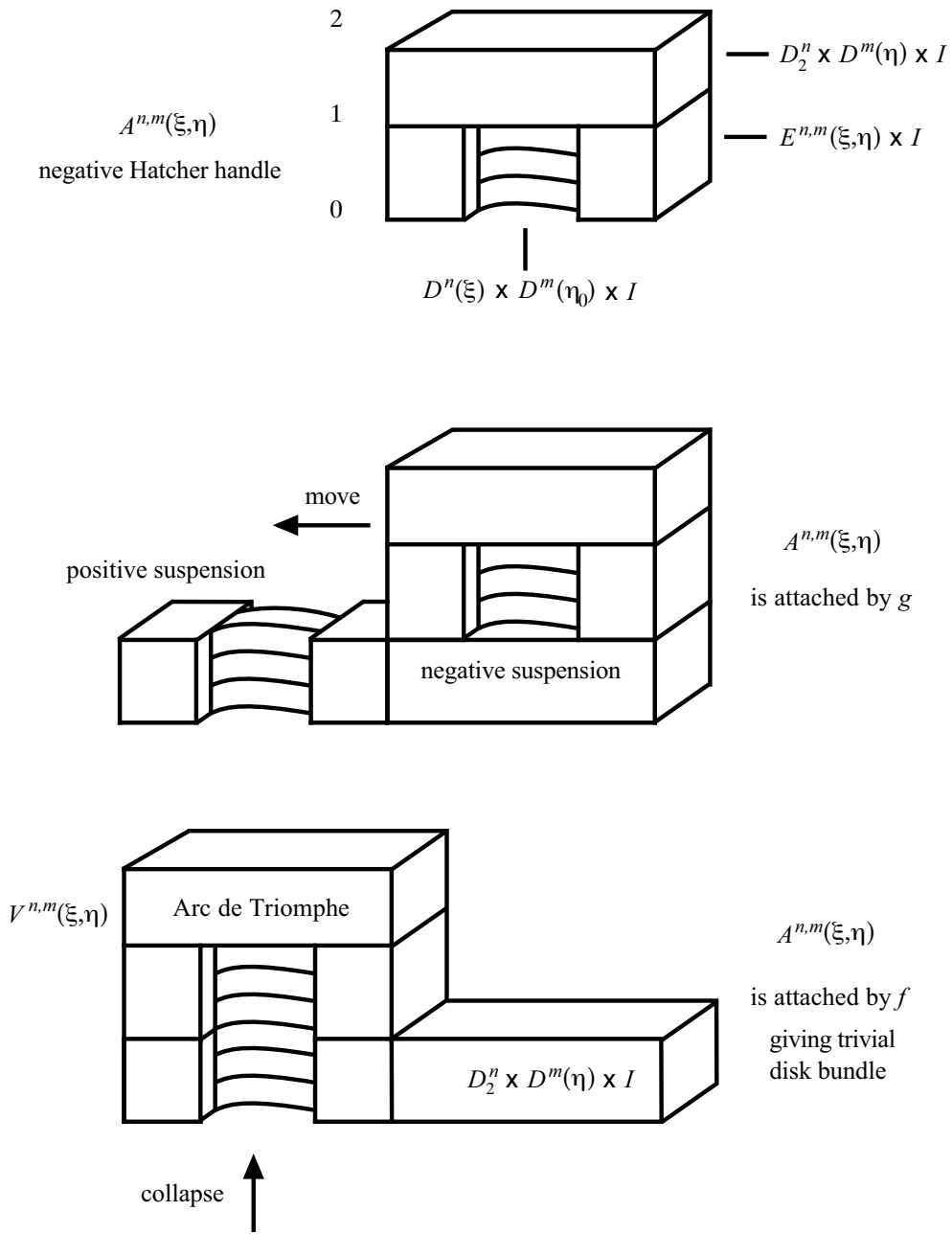}%[width=6in]{xxx.pdf}
\caption{Positive and negative Hatcher handles are cancelled using Arc de Triomphe}
\end{figure}
\end{center}
%

% subsection

%\newpage
%-----------------------------------------------------------------------------------
%            sub section {twisted version}
%-----------------------------------------------------------------------------------

\subsection{Twisted version}\label{subsecA22}

Remark \ref{rem:ch(xi) span H4k(B,d0)} above and the main theorem (Corollary 2.2.2) of \cite{Second} show that, rationally stably, all exotic smooth structures on trivial disk bundles are given by Hatcher's example. Now we consider nontrivial disk bundles.

Stably, it is easy to construct exotic smooth structures on nontrivial linear disk bundles. If we start with any vector bundle $\xi_0$ over $B$ which is trivial over $\d_0B$, we can take the associated disk bundle $D^N(\xi_0)$. The fiberwise product
\[
	D^N(\xi_0)\oplus E^{n,m}(\xi)
\]
with corners rounded is a smooth disk bundle fiberwise homeomorphic to $D^N(\xi_0)\times D^{n+m}$ with the same higher torsion as $E^{n,m}(\xi)$ since $IK$ torsion has the property that it is invariant under passage to linear disk bundles.

\begin{cor}
Given any linear disk bundle $D^N(\xi_0)$ over $B$ which is trivial over $\d_0B$, the collection of all stable smooth structures on $D^N(\xi_0)$ given by Hatcher's construction spans the vector space
\[
	\pi_0\std{B}{\d_0}(D^N(\xi_0))\otimes\RR\cong H^{4\bullet}(B,\d_0B)
\]
\end{cor}

Now we give the unstable version of the last corollary and use it to define ``Hatcher handles''. Suppose that $(B,\d_0 B)$ is a manifold pair as before with $\dim B=q$. Let $\xi,\eta$ be vector bundles over $B$ of dimension $n,m$ so that $\xi$ is trivial over $\d_0B$ and $J(\xi)=0\in J(B/\d_0B)$. As in Lemma \ref{embedding lemma} we have the following.

\begin{lem}\label{twisted embedding lemma}
If $m>n>q$ then there is a smooth fiberwise embedding of pairs:
\[
	j:(D^n(\xi),S^{n-1}(\xi))\to (D^n,S^{n-1})\times D^m(\eta)
\]
over $B$ which is a standard linear embedding over $\d_0B$ and which is transverse to $S^{n-1}\times D^m(\eta)$. Furthermore, if $m\ge q+3$ then this fiberwise embedding is unique up to fiberwise isotopy.
\end{lem}

Let $\eta_0$ be the unique $m$-plane bundle over $B$ so that $\xi\oplus\eta_0\cong\e^n\oplus\eta$ where $\e^n$ is the trivial $n$-plane bundle over $B$. Then the embedding given by the lemma thickens to a codimension 0 fiberwise embedding
\[
	(D(j),S(j)):(D^n(\xi),S^{n-1}(\xi))\oplus D^m(\eta_0)\hookrightarrow (D^n,S^{n-1})\times D^m(\eta)
\]
which is a standard linear embedding over $\d_0B$. Let $E^{n,m}(\xi,\eta)$ denote the $n+m$ disk bundle over $B$ given by
\[
	E^{n,m}(\xi,\eta)=
	D^n(\xi)\oplus D^m(\eta_0)
		\cup_{S(j)} S^{n-1}\times I\times D^{m}(\eta)
\]
with corners rounded. Up to fiberwise diffeomorphism, this is independent of the choice of $g$ if $m\ge q+3$. As before we have a fiberwise embedding $F(j):E^{n,m}(\xi,\eta)\into D^n\times D^m(\eta)$ and we can define the positive and negative suspensions of $E^{n,m}(\xi)$ to be
\[
	\s_+E^{n,m}(\xi,\eta)=E^{n,m}(\xi,\eta)\times I
\]
which is fiberwise diffeomorphic to $E^{n,m+1}(\xi,\eta)$ after corners are rounded and
\[
	\s_-E^{n,m}(\xi,\eta)=D^n\times D^m(\eta)\times [-1,0]\cup_{F(j)\times 0} E^{n,m}(\xi,\eta)\times I\cup_{F(j)\times 1} D^n\times D^m(\eta)\times[1,2]
\]
which is a model for $E^{n+1,m}(\xi,\eta)$.
As before, the Framing Principle implies that the higher \IK-torsion of this bundle is the normalized Chern character (Def. \ref{def: normalized Chern character}) of $\xi$:

\begin{thm}\label{torsion of twisted Hatcher disk bundle}
$E^{n,m}(\xi,\eta)$ is a smooth $n+m$ disk bundle over $B$ which is fiberwise diffeomorphic to the linear disk bundle $D^{n+m}(\eta)$ over $\d_0B$ and fiberwise homeomorphic to $D^{n+m}(\eta)$ over $B$. Furthermore,
\[
	\t^{\IK}(E^{n,m}(\xi,\eta))=(-1)^{n}\normch(\xi)\in H^{4\bullet}(B,\d_0B)
\]
\end{thm}

\begin{rem} This theorem can be stated as the commutativity of the following diagram:
\[
\xymatrix{
&G(B,\d_0B) \ar[rd]^{(-1)^n\normch}\ar[ld]_{E^n(-,\eta)} \\
\pi_0\std{B}{\d_0}(D(\eta))\ar[rr]^{\t^\IK}&& H^{4\bullet}(B,\d_0B)
	}%end xymatrix
\]
where $G(B,\d_0B)$ is the group of all homotopy classes of pointed maps $\xi:B/\d_0B\to G/O$. Here $E^n(-,\eta)$ is the map which sends $\xi$ to the direct limit of $E^{n,m}(\xi,\eta)$ as $m$ goes to $\infty$.
\end{rem}

Since the torsion of a linear disk bundle is trivial, the torsion of the disk bundle $E^{n,m}(\xi,\eta)$ is equal to the torsion of the $h$-cobordism bundle given by deleting a neighborhood of a section. The fiberwise boundary of $E^{n,m}(\xi,\eta)$ is a smooth $n+m-1$ dimensional sphere bundle over $B$ which is fiberwise tangentially homeomorphic to the linear sphere bundle $S^{n+m-1}(\eta)$.

% and the torsion of the disk bundle is half the difference torsion when $n+m-1$ is odd.

\begin{cor}\label{cor: torsion of twisted Hatcher sphere bundle}
Suppose that $n+m-1$ is odd. Then the vertical boundary $\d\vv E^{n,m}(\xi,\eta)$ of this disk bundle is a smooth sphere bundle which is fiberwise tangentially homeomorphic to the linear sphere bundle $S^{m+n-1}(\eta)$ and fiberwise diffeomorphic to this bundle over $\d_0B$ and the difference torsion is twice the normalized Chern character of $\xi$:
\[
	\t^{\IK}(\d\vv E^{n,m}(\xi,\eta),S^{n+m-1}(\eta))=(-1)^{n}2\normch(\xi)\in H^{4\bullet}(B,\d_0B)
\]
In particular, assuming that $\xi$ is rationally nontrivial, this gives an exotic smooth structure on $S^{n+m-1}(\eta)$.
\end{cor}

\begin{proof} For oriented sphere bundles, the absolute torsion is defined and the difference torsion is just the difference:
\[
	\t^\IK(\d\vv E^{n,m}(\xi,\eta),S^{n+m-1}(\eta))=\t^\IK(\d\vv E^{n,m}(\xi,\eta))-\t^\IK(S^{n+m-1}(\eta))
\]
Each term can be computed using the equation% from Remark \ref{rem:extension to not closed fibers}.%which is satisfied by any axiomatic higher torsion invariant $\t$ and any smooth bundle $E$ for which $E$, its vertical boundary $\d\vv E$ and vertical double $DE$ have well-defined higher torsion invariants \cite{I:Axioms0}.
\[
	\t(E)=\tfrac12\t(\d\vv E)+\tfrac12\t(DE)
\]
where $DE$ is the vertical double of $E$. (See Remark \ref{rem:extension to not closed fibers}.) If we take $E$ to be the linear disk bundle $E=D^{n+m}(\eta)$, then the triviality of the Igusa-Klein torsion for linear disk bundles implies that
\[
	\t^\IK(S^{n+m-1}(\eta))=-\t^\IK(S^{n+m}(\eta))
\]
If we take $E=E^{n,m}(\xi,\eta)$, then the fiberwise double $DE$, having closed even dimensional manifold fibers, has the same higher torsion as the linear sphere bundle $S^{n+m}(\eta)$ (by Theorem \ref{even torsion is a tangential invariant} since $n+m$ is even):
\[
	\t^\IK(\d\vv E^{n,m}(\xi,\eta))= 
	2\t^\IK(E^{n,m}(\xi,\eta))-\t^\IK(S^{n+m}(\eta))
\]
The relative torsion is the difference:
\[
	\t^\IK(\d\vv E^{n,m}(\xi,\eta))-\t^\IK(S^{n+m-1}(\eta))= 
	2\t^\IK(E^{n,m}(\xi,\eta))=(-1)^n2\normch(\xi)
\]
by Theorem \ref{torsion of twisted Hatcher disk bundle} above.
\end{proof}

% subsection

%\newpage
%-----------------------------------------------------------------------------------
%            sub section {Hatcher handles}
%-----------------------------------------------------------------------------------

\subsection{Hatcher handles}\label{ss: Hatcher handles}\label{subsecA23}

Suppose that $p:M\to B$ is a smooth manifold bundle whose fiber dimension is $N=n+m$ where $m>n>q$. Let $s:B\to M$ be a smooth section of $p$ with image in the fiberwise interior of $M$. Since $m=N-n>q+1$, the space of $n$ frames in $\RR^N$ is $q+1$-connected. So there exists a smooth fiberwise embedding $f:D^{n}\times B\to M$ equal to $s$ along the zero section and $f$ is uniquely determined up to isotopy by $s$. Let $\eta$ be the vertical normal bundle to the image of $f$ in $M$. This is the unique $m$ plane bundle over $B$ which is stably isomorphic to the pull back along $s$ of the vertical tangent bundle of $M$. Then $f$ extends to a fiberwise embedding
\begin{equation}\label{eq:D(s)}
	D(s):D^{n}\times D^{m}(\eta)\into M
\end{equation}
whose image is a tubular neighborhood of the image of the section $s$ and $D(s)$ is determined up to isotopy by $s$. We will use this embedding $D(s)$ to attach positively and negatively suspended Hatcher disk bundles to the top $M\times 1$ of the bundle $M\times I\to B$. We call these \emph{positive} and \emph{negative Hatcher handles}. We will also show that, when the negative Hatcher handle is attached on top of the positive Hatcher handle, they form the Arc de Triomphe which cancels.

To visualize these three situations, it may help to think of the positive Hatcher handle as two balls attached together on a string with one ball attached to the ground. This configuration is topologically contractible to its attachment point on the ground but not smoothly (Figure \ref{fig: positive Hatcher handle}). The negative Hatcher handle resembles the handle on a briefcase with a flexible membrane filling in the ``hole''. This is also topologically contractible to the base (the briefcase) but not smoothly (Figure \ref{fig: negative Hatcher handle}). The Arc de Triomphe resembles the hook on a coat hanger together with a semicircular membrane attached only to the curved part of the hook. This is smoothly contractible since the membrane smoothly deforms into the metal part and then the metal hook smoothly contracts to the base (Figure \ref{fig: cancelling Hatcher handles}). The term ``Arc de Triomphe'' may be misleading since this structure has one end up in the air and only the other end attached to the ground along a ``stem''. The Arc de Triomphe and negative Hatcher handles are diffeomorphic but they have different properties since the former is attached trivially and the latter is attached nontrivially to the base.

%-----------------------------------------------------------------------------------
%            subsub section {positive Hatcher handles}
%-----------------------------------------------------------------------------------

\subsubsection{Positive Hatcher handles}

Let $h_0:D_0^n\into S^{n-1}\times I$ be a fixed smooth embedding where $D_0^n=D^n$ is a copy of the standard $n$-disk. Taking the product with $D^m(\eta)$ we get a fiberwise embedding of $D_0^n\times D^m(\eta)$ into $E^{n,m}(\xi,\eta)$:
\[
	h=h_0\times id_{D^m(\eta)}:D_0^n\times D^m(\eta)\into S^{n-1}\times I\times D^m(\eta)\subseteq E^{n,m}(\xi,\eta)
\]
We define the {\bf positive Hatcher handle} to be the pair $(B^{n,m}(\xi,\eta),\d_0B^{n,m}(\xi,\eta))$ where
\[
	B^{n,m}(\xi,\eta)=D_0^n\times D^m(\eta)\times I\cup_{h\times1} E^{n,m}(\xi,\eta)\times [1,2]
\]
and $\d_0B^{n,m}(\xi,\eta)=D_0^n\times D^m(\eta)\times0$. We can attach $B^{n,m}(\xi,\eta)$ to $M\times I$ along any fiberwise embedding $D(s):\d_0B^{n,m}(\xi,\eta)\to M\times 1$ where $s:B\to M$ is a smooth section of $M$ as in \eqref{eq:D(s)} above. The result will be denoted:
\[
	E_+^{n,m}(M,s,\xi)=M\times I\cup_{D(s)}B^{n,m}(\xi,\eta)
\]

Since the bundle pair $(B^{n,m}(\xi,\eta),\d_0B^{n,m}(\xi,\eta))$ is fiberwise homeomorphic to the disk bundle pair $D^n\times D^m(\eta)\times (I,0)$, the bundle $E_+^{n,m}(M,s,\xi)$ is fiberwise homeomorphic to the bundle $M\times I$. However, $E_+^{n,m}(M,s,\xi)$ is a smooth bundle (when corners are rounded) whose fibers are $h$-cobordisms.

\begin{figure}[ht]
\begin{center}
%
%\vs5
{
\setlength{\unitlength}{3cm}
%\centerline
{\mbox{
\begin{picture}(3.5,1.5)
      \thicklines
%    \thinlines
\put(.5,.5){\line(1,0){2}
}
\put(2.95,.5){\line(1,0){.45}
}
\put(0,0){
\qbezier(0,0)(.25,.25)(.5,.5)
\line(1,0){3}
%\qbezier(0,0)(2,0)(3,0)
}
\put(3,0){
\qbezier(0,0)(.2,.25)(.4,.5)
}
\put(.5,.1){$M\times 1$
}
\put(1.2,0){  \put(1.1,.7){
  \line(1,0){.5}
  }
  \put(1.1,.7){
  \line(0,1){.5}   \qbezier(0,.5)(.1,.6)(.2,.7)   
  }
  \put(1.6,1.2){
   \line(0,-1){.5} 
   \qbezier(0,0)(.1,.1)(.2,.2)
  }
  \put(1.1,1.2){
   \line(1,0){.5} 
  }
  \put(1.3,1.4){
   \line(1,0){.5} 
  }
    \put(1.8,.9){
  \line(0,1){.5}   \qbezier(0,0)(-.1,-.1)(-.2,-.2)   
  }}
  %  %
  \put(1.1,.7){
  \line(1,0){.5}
  }
  \put(1.1,.7){
  \line(0,1){.5}   \qbezier(0,.5)(.1,.6)(.2,.7)   
  }
  \put(1.6,1.2){
   \line(0,-1){.5} 
   \qbezier(0,0)(.1,.1)(.2,.2)
  }
  \put(1.1,1.2){
   \line(1,0){.5} 
  }
  \put(1.3,1.4){
   \line(1,0){.5} \qbezier(0,0)(0,-.01)(0,-.06)
  }
  \put(1.7,1.25){
  \qbezier(0,0)(.3,-.1)(.75,0.05)
  \line(0,-1){.5}
  }
  \put(1.8,1.35){
  \qbezier(0,0)(.2,-.1)(.75,0.05)
  }
  \put(1.7,.75){
  \qbezier(0,0)(.2,-.1)(.65,0.05) %xxx
  }
     \put(1.65,.7){
   \qbezier(0,0)(.02,.02)(.05,.05)
  }
  \put(2.5,.3){\line(1,0){.3}}
  \put(2.5,.3){\line(0,1){.4}}
  \put(2.8,.3){\line(0,1){.4}}
  \put(2.8,.3){
     \qbezier(0,0)(.1,.1)(.15,.15)
  }
  \put(2.9,.45){
  \line(0,1){.35}
  }
  \put(-.1,1){
  $E^{n,m}(\xi,\eta)\times [1,2]$
  }
  \put(3,.6){
  $D_0^n\times D^m(\eta)\times I$
  }
\end{picture}}
}}
%\vs5
\caption{(Positive Hatcher handle) The positive suspension $\s_+E^{n,m}(\xi,\eta)$ is attached to the top $M\times 1$ of $M\times I$ by the ``stem'' $D_0^n\times D^m(\eta)\times I$.}
\label{fig: positive Hatcher handle}
\end{center}
\end{figure}
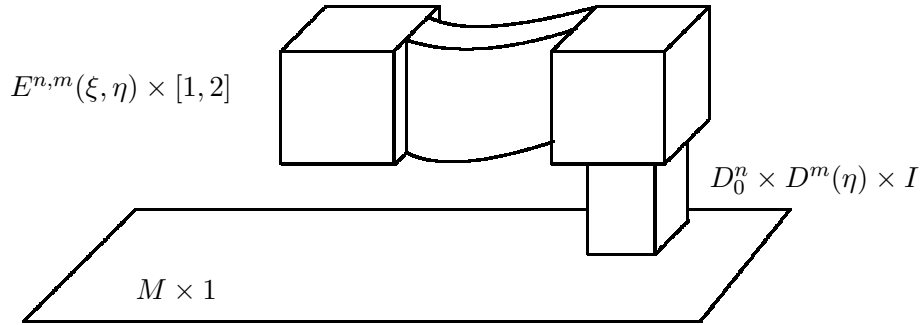

\begin{thm}
Let $T$ be a closed fiberwise tubular neighborhood of $s(B)$ in $M$. Then there is a fiberwise homeomorphism $M\times I\to E_+^{n,m}(M,s,\xi)$ which is the identity (and thus a diffeomorphism) on $M\times 0$ and a diffeomorphism on the closure of $(M-T)\times I$. Furthermore the difference torsion is the same as the IK-torsion of $E^{n,m}(\xi,\eta)$:
\[
	\t^\IK(E_+^{n,m}(M,s,\xi),M\times I) =\t^\IK(E^{n,m}(\xi,\eta))=(-1)^{n}\normch(\xi)\in H^{4\bullet}(B,\d_0B)
\]
\end{thm}

\begin{rem} This theorem can be viewed as the commutativity of the diagram:
\[
\xymatrix{
&G(B,\d_0B) \ar[rd]^{(-1)^n\normch}\ar[ld]_{E^n(-,\eta)}\ar[d]^(.65){E_+^n(M,s,-)} \\
\pi_0\std{B}{\d_0}(D(\eta))\ar[r]_{s_\ast}\ar@/_2pc/[rr]^{\t^\IK}&  \pi_0\std{B}{\d_0}(M)\ar[r]_(.4){\t^\IK}& H^{4\bullet}(B,\d_0B)
	}%end xymatrix
\]
\end{rem}

Let $M'=\d_1E_+^{n,m}(M,s,\xi)$ be the top boundary of the $h$-cobordism bundle $E_+^{n,m}(M,s,\xi)$. 

\begin{cor} $M'$ is fiberwise tangentially homeomorphic to $M$ and, if the fiber dimension $N=n+m$ of $M'$ is odd, then the relative \IK-torsion is equal to twice the normalized Chern character of $\xi$:
\[
	\t^\IK(M',M)=(-1)^{n}2\normch(\xi)\in H^{4\bullet}(B,\d_0B)
\]
\end{cor}

%-----------------------------------------------------------------------------------
%            subsub section {negative Hatcher handles}
%-----------------------------------------------------------------------------------

\subsubsection{Negative Hatcher handles}

The {\bf negative Hatcher handle} is defined to be the pair $(A^{n,m}(\xi,\eta),\d_0A^{n,m}(\xi,\eta))$ where
\[
	A^{n,m}(\xi,\eta)=E^{n,m}(\xi,\eta)\times I\cup_{F(j)\times 1}
	D_2^n\times D^m(\eta)\times [1,2]
\]
and $\d_0A^{n,m}(\xi,\eta)=E^{n,m}(\xi,\eta)\times 0$. When we attach this to the top of $M\times I$ using the composite map
\[
	E^{n,m}(\xi,\eta)\xrarrow{F(j)}D_2^n\times D^m(\eta)\xrarrow{D(s)}M
\]
we denote the result by
\[
	E_-^{n,m}(M,s,\xi)=M\times I\cup_{D(s)\circ F(j)}A^{n,m}(\xi,\eta)
\]

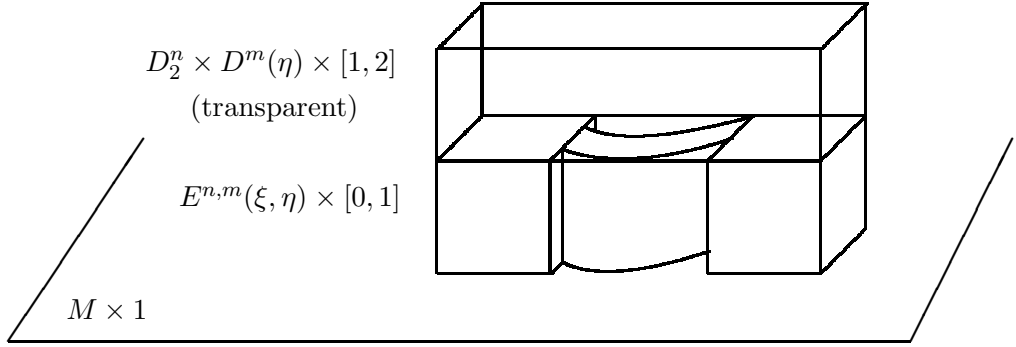
\begin{figure}[ht]
\begin{center}
%
%\vs5
{
\setlength{\unitlength}{3cm}
%\centerline
{\mbox{
\begin{picture}(3.5,2)
      \thicklines
%    \thinlines
  %
\put(1.2,0){  
\put(1.1,.7){ 
  \line(1,0){.5}
  }
  \put(1.1,.7){
  \line(0,1){.5}   \qbezier(0,.5)(.1,.6)(.2,.7)   
  }
  \put(1.6,1.2){ 
   \line(0,-1){.5} 
   \qbezier(0,0)(.1,.1)(.2,.2)
  }
  \put(1.6,1.2){ 
   \line(0,1){.5} 
  }  
  \put(1.65,1.7){ 
   \qbezier(0,0)(.1,.1)(.2,.2)
  }  
  \put(1.1,1.2){
   \line(1,0){.5} 
  }
  \put(1.3,1.4){
   \line(1,0){.5} 
  }
    \put(1.8,.9){
  \line(0,1){.5}   \qbezier(0,0)(-.1,-.1)(-.2,-.2)   
  }
     \put(1.8,.9){
  \line(0,1){1}   
  }
  }
  %  %
  \put(1.1,.7){ % base E x 0
  \line(1,0){.5}
  }
  \put(1.1,.7){
  \line(0,1){.5}   \qbezier(0,.5)(.1,.6)(.2,.7)   
  }
    \put(1.15,1.2){
 \qbezier(0,.5)(.1,.6)(.2,.7)   
 } 
  \put(1.3,1.9){
  \line(1,0){1.7}    
  }
  \put(1.3,1.9){
  \line(0,-1){.5}    
  }
  \put(1.1,.7){
  \line(0,1){1}    
  }
  \put(1.1,1.7){
  \line(1,0){1.7}
  }
  \put(1.6,1.2){
   \line(0,-1){.5} 
   \qbezier(0,0)(.1,.1)(.2,.2)
  }
  \put(1.1,1.2){
   \line(1,0){1.5} 
  }
  \put(1.3,1.4){
   \line(1,0){.5} 
   \qbezier(0,0)(0,-.01)(0,-.06)
  }
    \put(1.8,1.4){
   \line(1,0){1} 
}
  \put(1.7,1.25){
  \qbezier(0,0)(.3,-.1)(.75,0.05)
  \line(0,-1){.5}
  }
  \put(1.8,1.35){
  \qbezier(0,0)(.2,-.1)(.75,0.05)
  }
  \put(1.7,.75){
  \qbezier(0,0)(.2,-.1)(.65,0.05)
  }
    \put(1.65,.7){
   \qbezier(0,0)(.02,.02)(.05,.05)
  } % end of base?
  \put(0,1){$E^{n,m}(\xi,\eta)\times [0,1]$}
  \put(-.15,1.6){$D_2^n\times D^m(\eta)\times [1,2]$}
  \put(.05,1.4){(transparent)}
  %  % begin base (M x 1)
  \put(-.8,.4){  
  \line(1,0){4}
  }
  \put(-.8,.4){  
  \line(2,3){.6}
  }
  \put(3.2,.4){  
  \line(1,2){.45}
  }
  \put(-.5,.5){$M\times 1$}
  \put(1.1,.7){
  \line(1,0){.5}
  }	% end base
\end{picture}}
}}
%\vs5
\caption{(Negative Hatcher handle) $A^{n,m}(\xi,\eta)$ is attached to the top $M\times 1$ of $M\times I$ along its base $E^{n,m}(\xi,\eta)\times 0$.}
\label{fig: negative Hatcher handle}
\end{center}
\end{figure}

The negative Hatcher handle is shown in Figure \ref {fig: negative Hatcher handle} and also in the top figure in Figure \ref{AdT figure} where $A^{n,m}(\xi,\eta)=E_1$. 

\begin{lem} When we attach the negative suspension of $E^{n,m}(\xi,\eta)$ to the top of $M\times I$ along the map $D(s):D_2^n\times D^m(\eta)\times 1\to M\times1$, the result
\[
	M\times I\cup_{D(s)} \s_-E^{n,m}(\xi,\eta)
\]
is fiberwise diffeomorphic to $E_-^{n,m}(M,s,\xi)$ with higher difference torsion given by
\[
	\t^\IK(E_-^{n,m}(M,s,\xi),M\times I) =-\t^\IK(E^{n,m}(\xi,\eta))=(-1)^{n+1}\normch(\xi)\in H^{4\bullet}(B,\d_0B)
\]
\end{lem}

\begin{proof}
When we attach $D_2^n\times D^m(\eta)\times [1,2]\subseteq \s_-E^{n,m}(\xi,\eta)$ to $M\times 1\subseteq \d\vv M\times I$ using the map $D(s):D_2^n\times D^m(\eta)\times 1\to M\times 1$, the result is fiberwise diffeomorphic to $M\times I$:
\[
	M\times I\cup_{D(s)} D_2^n\times D^m(\eta)\times [1,2]\cong M\times I
\]
since we can pull $D_2^n\times D^m(\eta)\times I$ into $M\times I$ by the trivial Lemma \ref{first trivial lemma}. Therefore,
\[
	M\times I\cup_{D(s)} \s_-E^{n,m}(\xi,\eta)=M\times I\cup_{D(s)} D_2^n\times D^m(\eta)\times [1,2]\cup_{F(j)} A^{n,m}(\xi,\eta)
\] 
is fiberwise diffeomorphic to $E_-^{n,m}(M,s,\xi)=M\times I\cup_{D(s)\circ F(j)} A^{n,m}(\xi,\eta)$.

The higher torsion calculation follows from the relative handlebody lemma (Remark \ref{rem:handlebody lemma}).
%\[	\t^\IK(M\times I\cup_{D(s)} \s_-E^{n,m}(\xi,\eta),M\times I)=\t^\IK( \s_-E^{n,m}(\xi,\eta))=(-1)^{n+1}\normch(\xi)\]
\end{proof}

% subsection

%-----------------------------------------------------------------------------------
%            subsub section {cancellation of Hatcher handles}
%-----------------------------------------------------------------------------------

\subsubsection{Cancellation of Hatcher handles}

We will take the ``union'' of the two constructions given above and attach both positive and negative Hatcher handles along the same section $s:B\to M$ and show that they cancel. As before, we have a smooth embedding
\[
	D(s):D^n\times D^m(\eta)\to M
\]
whose image is a tubular neighborhood of $s(B)$. Inside this disk bundle we create two smaller isomorphic disk bundles using embedding:
\[
	j_+,j_-:D^n\times D^m(\eta)\to D^n\times D^m(\eta)
\]
given by $j_+(x,y)=(\tfrac13(x+e_n),y)$ where $e_n$ is the last unit vector of $D^n$ and $j_-(x,y)=(\tfrac13(x-e_n),y)$. Since they are less than half as wide, these two embeddings are disjoint. Suppose that $E^{n,m}(\xi,\eta)$ is a Hatcher disk bundle as in the construction above. We first attach the positive Hatcher handle $B^{n,m}(\xi,\eta)$ along its base $\d_0B^{n,m}(\xi,\eta)=D^n\times D^m(\eta)\times0$ to the top $M\times 1$ of $M\times I$ using the fiberwise embedding $D(s)\circ j_-$. Next we attach the negative Hatcher handle $A^{n,m}(\xi,\eta)$ to the top of $M\times I$ along its base $\d_0A^{n,m}(\xi,\eta)=E^{n,m}(\xi,\eta)$ using the composite map
\[
	E^{n,m}(\xi,\eta)\xrarrow{F(j)}D^n\times D^m(\eta)\xrarrow{ j_+} D^n\times D^m(\eta)\xrarrow{D(s)}M
\]

Let $T$ be the image of $D(s)$ with corners rounded. Thus $T$ is a $D^{n+m}$-bundle over $B$. Let $S=\d\vv T$ be the fiberwise boundary of $T$. This is a sphere bundle over $B$. After attaching the positive and negative Hatcher handles to the top of $M\times I$ we get a new bundle
\[
	W=M\times I\cup_{D(s)\circ j_-} B^{n,m}(\xi,\eta)\cup_{D(s)\circ j_+\circ F(j)}A^{n,m}(\xi,\eta)
\]
Note that since $B^{n,m}(\xi,\eta)$ and $A^{n,m}(\xi,\eta)$ are both attached in the interior of $T$, this new bundle is the union of $C\times I$ and $T\times I\cup B\cup A$ where $C$ is the closure of $M-T$ and $A,B$ denote the Hatcher handles.

\begin{figure}[ht]
\begin{center}
%
%\vs5
{
\setlength{\unitlength}{3cm}
%\centerline
{\mbox{
\begin{picture}(3.5,2.5)
      \thicklines
%    \thinlines
\put(.5,.5){\line(1,0){2}
}
\put(2.95,.5){\line(1,0){.45}
}
\put(0,0){
\qbezier(0,0)(.25,.25)(.5,.5)
\line(1,0){3}
%\qbezier(0,0)(2,0)(3,0)
}
%\thicklines
\put(0,0.5){  %begin top
      \put(1.3,1.4){ 
   \line(1,0){1.5} 
}
  \put(1.1,1.2){ 
  \line(0,1){.5}   \qbezier(0,.5)(.1,.6)(.2,.7)   
  }
  \put(2.8,1.2){ 
  \line(0,1){.5}   \qbezier(0,.5)(.1,.6)(.2,.7)   
  }
  \put(3,1.4){
  \line(0,1){.5}
  }
  \put(1.3,1.4){
  \line(0,1){.5}
  }
  \put(1.3,1.9){
  \line(1,0){1.7}
  }
  \put(1.1,1.7){
  \line(1,0){1.7}
  }
  \put(1.1,1.2){
  \line(1,0){1.7}
  }
}  % end top
%\thinlines
\put(3,0){
\qbezier(0,0)(.2,.25)(.4,.5)
}
\put(.5,.1){$M\times 1$
}
\put(1.2,0){  
  \put(1.1,.7){
  \line(1,0){.5}
  }
  \put(1.1,.7){ 
  \line(0,1){.5}   \qbezier(0,.5)(.1,.6)(.2,.7)   
  }
  \put(1.6,1.2){
   \line(0,-1){.5} 
   \qbezier(0,0)(.1,.1)(.2,.2)
  }
  \put(1.1,1.2){
   \line(1,0){.5} 
  }
  \put(1.3,1.4){
   \line(1,0){.5} 
  }
    \put(1.8,.9){
  \line(0,1){.5}   \qbezier(0,0)(-.1,-.1)(-.2,-.2)   
  }}
\put(1.2,0.5){  
  \put(1.1,.7){
  \line(1,0){.5}
  }
  \put(1.1,.7){ 
  \line(0,1){.5}   \qbezier(0,.5)(.1,.6)(.2,.7)   
  }
  \put(1.6,1.2){
   \line(0,-1){.5} 
   \qbezier(0,0)(.1,.1)(.2,.2)
  }
  \put(1.1,1.2){
   \line(1,0){.5} 
  }
  \put(1.3,1.4){
   \line(1,0){.5} 
  }
    \put(1.3,1.4){
   \line(0,-1){.5} 
  }
    \put(1.8,.9){
  \line(0,1){.5}   \qbezier(0,0)(-.1,-.1)(-.2,-.2)   
  }}
  %  %
  \put(1.1,.7){
  \line(1,0){.5}
  }
  \put(1.1,.7){
  \line(0,1){.5}   \qbezier(0,.5)(.1,.6)(.2,.7)   
  }
  \put(1.6,1.2){
   \line(0,-1){.5} 
   \qbezier(0,0)(.1,.1)(.2,.2)
  }
  \put(1.6,1.7){
   \line(0,-1){.5} 
   \qbezier(0,0)(.1,.1)(.2,.2)
  }
  \put(1.1,1.2){
   \line(1,0){.5} 
  }
  \put(1.1,1.2){
   \line(0,1){.5} 
  }
  \put(1.3,1.4){
  \line(1,0){.5} \qbezier(0,0)(0,-.01)(0,-.06)
  }
  \put(1.7,1.25){
  \qbezier(0,0)(.3,-.1)(.75,0.05)
  \line(0,-1){.5}
  }
  \put(1.8,1.35){
  \qbezier(0,0)(.2,-.1)(.75,0.05)
  }
  \put(0,.5){
    \put(1.7,1.25){
  \qbezier(0,0)(.3,-.1)(.75,0.05)
  \line(0,-1){.5}
  }
  \put(1.8,1.35){
  \qbezier(0,0)(.2,-.1)(.75,0.05)
  }
  }
  \put(1.7,.75){
  \qbezier(0,0)(.2,-.1)(.65,0.05) %xxx
  }
     \put(1.65,.7){
   \qbezier(0,0)(.02,.02)(.05,.05)
  }
  \put(2.5,.3){\line(1,0){.3}}
  \put(2.5,.3){\line(0,1){.4}}
  \put(2.8,.3){\line(0,1){.4}}
  \put(2.8,.3){
     \qbezier(0,0)(.1,.1)(.15,.15)
  }
  \put(2.9,.45){
  \line(0,1){.35}
  }
        \put(1.8,1.4){ 
   \line(1,0){1} 
}
      \put(1.6,1.2){ 
   \line(1,0){1} 
}
    \put(1.34,1.4){\line(0,1){.5}
}
    \put(1.84,1.4){\line(0,1){.5}
}
    \put(1.8,1.36){\line(0,1){.5}
}
    \put(2.45,1.31){\line(0,1){.5}
}
\put(1.14,1.7){       \qbezier(0,0)(.1,.1)(.2,.2)
}
\put(0,.2){
  \put(0.15,1.6){$A^{n,m}(\xi,\eta)$}
  \put(.05,1.4){(transparent)}
  }
    \put(0.15,.9){$B^{n,m}(\xi,\eta)$}
\end{picture}}
}}
%\vs5
\caption{(Arc de Triomphe) The negative Hatcher handle $A^{n,m}(\xi,\eta)$ is attached on top of the positive Hatcher handle $B^{n,m}(\xi,\eta)$ forming the Arc de Triomphe $V^{n,m}(\xi,\eta)$ which is diffeomorphic to $A^{n,m}(\xi,\eta)$ but attached to the top $M\times 1$ of $M\times I$ on the ``stem'' $D_0^n\times D^m(\eta)\times I$.}
\label{fig: cancelling Hatcher handles}
\end{center}
\end{figure}
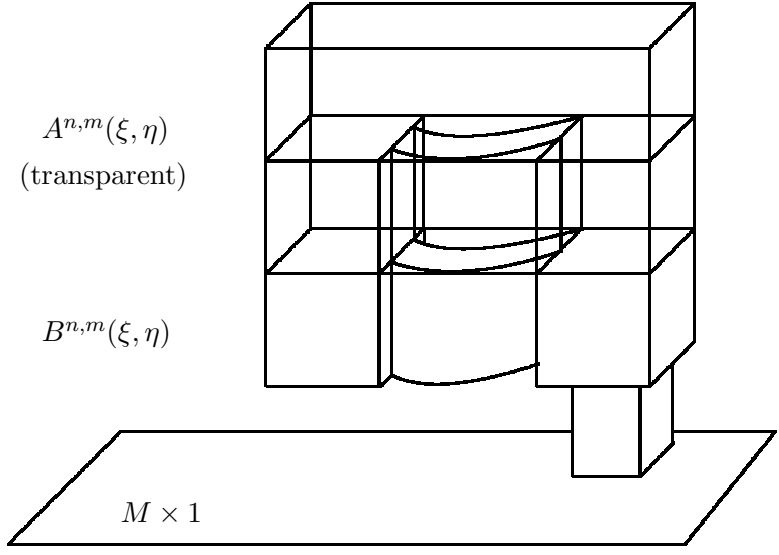

\begin{prop}[second cancellation lemma]\label{second AdT cancellation lemma}
$W$ is fiberwise diffeomorphic to $M\times I$ after rounding corners and this diffeomorphism is the identity on $C\times I$ and on $M\times 0$.
\end{prop}

\begin{proof}
The argument is almost the same as in Proposition \ref{basic AdT cancellation lemma}. Since $\d_0A^{n,m}(\xi,\eta)=E^{n,m}(\xi,\eta)$ is a disk bundle attached using the same tangential data as $B^{n,m}(\xi,\eta)$, there is an isotopy of the attaching map ${D(s)\circ j_+\circ F(j)}$ of the negative Hatcher handle $A^{n,m}(\xi,\eta)$ to the mapping
\[
	E^{n,m}(\xi,\eta)\to E^{n,m}(\xi,\eta)\times 1\subset (E^{n,m}(\xi,\eta)\cup D_0^n\times D^m(\eta))\times I=B^{n,m}(\xi,\eta)
\]
placing $A^{n,m}(\xi,\eta)$ onto the top sides $E^{n,m}(\xi,\eta)\times 1$ of the positive Hatcher handle $B^{n,m}(\xi,\eta)=E^{n,m}(\xi,\eta)\cup D_0^n\times D^m(\eta)\times I$. After moving the attaching map, $A^{n,m}(\xi,\eta)$ is attached on top of $E^{n,m}(\xi,\eta)\times I$ and their union is 
\[
	V^{n,m}(\xi,\eta)=E^{n,m}(\xi,\eta)\times I\cup A^{n,m}(\xi,\eta)=E^{n,m}(\xi,\eta)\times [0,2]\cup D_2^n\times D^m(\eta)\cong A^{n,m}(\xi,\eta)
\]
which is attached on $M\times 1$ along the image of $D(s)\circ j_-$ by the ``stem'' $D_0^n\times D^m(\eta)$. By Lemma \ref{first trivial lemma}, $V^{n,m}(\xi,\eta)\cup D_0^n\times D^m(\eta)$ is fiberwise diffeomorphic to $D_2^n\times D^m(\eta)\cup D_0^n\times D^m(\eta)$. This is a linear disk bundle and, therefore, attaching this to the top of $T\times I$ gives a bundle $X$ diffeomorphism of $T\times I$ fixing $S\times I$. This sequence of deformations and diffeomorphisms gives a diffeomorphism $T\times I\cup B\cup A\cong T\times I$ which is the identity on $S\times I$ and therefore, can be pasted with $C\times I$ to give a fiberwise diffeomorphism $W=C\times I\cup T\times I\cup B\cup A\cong M\times I$ as claimed.
\end{proof}

% subsection

%\newpage
%-----------------------------------------------------------------------------------
%            sub section {immersed Hatcher handles}
%-----------------------------------------------------------------------------------

\subsection{Immersed Hatcher handles}\label{subsecA24}

Since ``Hatcher handles'' are attached in a neighborhood of one point, several of them can be attached at different points at the same time. And, in the AdT construction, there are necessarily two Hatcher handles attached to the same fiber.

Let $L$ be a $q$ manifold with boundary $\d L=\d_0 L\cup \d_1 L$ where $\d_0 L,\d_1 L$ are $q-1$ manifolds meeting along their common boundary. {Let $\ll:L\to B$ be an immersion so that $\ll^{-1}(\d_1 B)=\d_1 L$ and let $\tilde\ll:L\to M$ be an embedding over $\ll$. } Then the immersed Hatcher handle construction will modify the smooth structure of $M$ in a neighborhood of the image of $\tilde\ll$. The reason that $\ll:L \to B$ will be an immersion and not an embedding is because, in the proof of key result, we will start with an Arc de Triomphe construction and separate the positive and negative Hatcher handles into immersed Hatcher handles. Since the AdT construction requires two handles to be attached over the same point in $B$, the mapping $\ll:L\to B$ parametrizing the separate handles will be 2 to 1 near these points. So, we cannot assume that $\ll$ is an embedding.

Suppose as before that $m>n>q$ and let
\[
	D(\tilde\ll):D_2^n\times D^m(\eta)\into M
\]
be a smooth embedding over $\ll:L\to B$ where $\eta$ is the pull-back along $\tilde\ll:L\to M$ of the stable vertical tangent bundle of $M$. As before, $D_2^n$ is the disk of radius 2 in $\RR^n$.

Let $\xi$ be an $n$-plane bundle over $L$ which is trivial over $\d_1L$ so that $J(\xi)=0\in J(L/\d_1L)$ and let $\eta_0$ be the unique $m$-plane bundle over $L$ so that $\xi\oplus\eta_0\cong \eta$. We define $W=E_+^{n,m}(M,\tilde\ll,\xi)$ to be the smooth $h$-cobordism bundle over $B$ so that $\d_0W=M$ given by
\[
	E_+^{n,m}(M,\tilde\ll,\xi)=M\times I\ \cup_{D(\tilde\ll)\circ F(j)} B^{n,m}(\xi,\eta)
\]
where $B^{n,m}(\xi,\eta)$ is the positive Hatcher handle parametrized by $L$. {This Hatcher handle will be ``tapered off'' along $\d_0L$ by which we mean (in the case when $\ll$ is an embedding) that we construct a fiberwise diffeomorphism $E_+^{n,m}(M,\tilde\ll,\xi)\cong M\times I$ over $\ll(\d_0L)$. 

When $\ll$ is an immersion, there will be points $b\in B$ so that $\ll^{-1}(b)$ contains more than one point. I.e., more than one Hatcher handle will be attached to the fiber $M_b$ of $M$ over $b$. In this case, we will delete those handles corresponding to the elements of $\d_0L$. By ``tapering off'' we mean that we will make this deletion operation smooth with respect to $b\in B$. To do this, we ``dig a hole'' underneath the Hatcher handle. The idea is the the ``hole'' is perfectly cylindrical, but we fill it will a deformed plug (the Hatcher handle). Over a neighborhood of $\d_0L$, the Hatcher handle is fiberwise diffeomorphic to the trivial disk bundle. So, over these points, the plug will fit perfectly into the hole and the result will be that fewer holes will be noticeably refilled. When $b$ moves around $B$ and the number of inverse image points in $L$ varies, this trick will make the transition smooth.}

{First we note that the smooth disk bundle over $L$ given by
\[
	E_L^{n,m+1}(\xi,\eta)=D_2^n\times D^m(\eta)\times I\cup_{F(j)} B^{n,m}(\xi,\eta)
\]
is fiberwise diffeomorphic to $D_2^n\times D^m\times I$ over a small neighborhood of $\d_0L$. We choose such a diffeomorphism.} Let $T$ be the image of $D(\tilde\ll):D_2^n\times D^m(\eta)\to M$. So $T\times I\subseteq M\times I$ is fiberwise diffeomorphic to $D^n\times D^m(\eta)\times I$. (In the analogy, $T\times I$ is the cylindrical chunk of dirt we pull out of the ``ground'' $M\times I$ creating a cylindrical hole: $(M-T)\times I$. We now fill the hole with $E_L^{n,m+1}(\xi,\eta)$ which is equivalent to $T\times I$ near $\d_0L$ by the chosen fiberwise diffeomorphism.) The smooth $h$-cobordism bundle $E_+^{n,m}(M,\tilde\ll,\xi)$ is given by:
\[
	E_+^{n,m}(M,\tilde\ll,\xi)=(M-T)\times I\cup E_L^{n,m+1}(\xi,\eta)
\]

\begin{thm}[torsion of immersed Hatcher handle]\label{torsion of immersed Hatcher}
The higher \IK-difference torsion of this bundle with respect to $M\times I$ is the image under the mapping
\[
	\ll_\ast:H^{4\bullet}(L,\d_0L)\cong H_{q-4\bullet}(L,\d_1L)\to H_{q-4\bullet}(B,\d_1B)\cong H^{4\bullet}(B,\d_0B)
\]
of the normalized Chern character of $\xi$:
\[
	\t^\IK (E_+^{n,m}(M,\tilde\ll,\xi),M\times I)=
	\ll_\ast\left((-1)^n\normch(\xi)\right)\in H^{4\bullet}(B,\d_0B;\RR)
\]
\end{thm}

\begin{rem} This theorem can be viewed as the commutativity of the diagram:
\[
\xymatrix{
G(L,\d_0L)\ar[r]^{E_L^n(-,\eta)}\ar[dr]_{E_+^n(M,\tilde\ll,-)}\ar@/^2pc/[rr]^{(-1)^n\normch} &  
\pi_0\std{L}{\d_0}(D(\eta))\ar[d]^{D(\tilde\ll)_\ast} \ar[r]_{\t^\IK}&
 H^{4\bullet}(L,\d_0L;\RR)\ar[d]^{\ll_\ast}\\
\qquad\qquad\qquad\qquad&  \pi_0\std{B}{\d_0}(M)\ar[r]_{\t^\IK}&
H^{4\bullet}(B,\d_0B;\RR)
	}%end xymatrix
\]
The commutativity of the upper curved triangle is Theorem \ref {torsion of twisted Hatcher disk bundle}.

%We also note that, if $\ll_\ast$ is replaced by the induced map in cohomology $\ll^\ast:H^\ast(B,\d_0B)\to H^\ast(L,\d_0 L)$, the diagram no longer commutes as we can easily see in the special case when $L$ is a disjoint union of two copies of $B$. Therefore, the use of homology instead of cohomology is essential.
\end{rem}

To prove this, we need to recall the precise statement of the Framing Principle from \cite{I:ComplexTorsion}. Suppose that $W\to B$ is a smooth $h$-cobordism bundle with fiberwise boundary equal to
\[
	\d\vv W=M\cup \d\vv M\times I\ \cup M_1
\]
and $f:W\to I$ is a fiberwise generalized Morse function equal to $0$ on $M$ and $1$ on $M_1$ and equal to projection to $I$ on $\d\vv M\times I$. Suppose that the fiberwise singular set $\Sig(f)$ of $f$ does not meet $W_{\d_0B}$. In particular, $W_{\d_0B}\cong M_{\d_0B}\times I$. We are in the restricted case when the birth death points of $f$ are {\bf framed} in the sense that the negative eigenspace bundle of $D^2f$ is trivial over the birth-death points. This implies that, over the set $\Sig_i(f)$ of Morse points of $f$ of index $i$, the negative eigenspace bundle of $D^2f$ is trivial along $\d_0\Sig_i(f)$ which is equal to the set of birth-death points to which $\Sig_i(f)$ converges. The Framing Principle was proved in this restricted case in \cite{I:BookOne}.

In general, the negative eigenspace bundle is a well defined stable vector bundle $\xi=\xi(f)$ on the entire singular set $\Sig(f)$. It is defined as follows. {At each index $i$ critical point $x$ of $f$ let $\xi(x)=\xi_i(x)\oplus \e^{N-i}$ where $\xi_i(x)$ is the $i$-dimensional negative eigenspace of $D^2f$ and  $\e^{N-i}$ is the trivial bundle with dimension $N-i$ where $N=n+m+1$ is the dimension of the fiber of $W\to B$.} This defines an $N$-plane bundle over $\Sig_i(f)$. At each cubic point we identify the positive cubic direction with the positive first coordinate direction in $\e^{N-i}$. This has the effect of pasting together these $N$-plane bundles over $\Sig_i(f)$ and $\Sig_{i+1}(f)$ along their common boundary for each $i$. The result is an $N$-plane bundle over all of $\Sig(f)$.

The projection mapping $p:(\Sig(f),\d\Sig(f))\to (B,\d_1B)$ induces a map in cohomology using Poincar\'e duality assuming that $B$ is oriented. (If $B$ is not oriented then just replace it with the disk bundle of the orientation line bundle.)
\[
	p^\Sig_\ast:H^\ast(\Sig(f))\cong H_{q-\ast}(\Sig(f),\d\Sig(f))\to
	H_{q-\ast}(B,\d_1B)\cong H^\ast(B,\d_0B)
\] 
Similarly, for each index $i$ we have the push-down operator:
\[
	p_\ast:H^\ast(\Sig_i(f),\d_0\Sig_i(f))\cong H_{q-\ast}(\Sig_i(f),\d_1\Sig_i(f))\to
	H_{q-\ast}(B,\d_1B)\cong H^\ast(B,\d_0B)
\]
where $\d_1\Sig_i(f)=\Sig_i(f)\cap\d\Sig(f)$ and $\d_0\Sig_i(f)$ is the set of birth-death points in the closure of $\Sig_i(f)$. We use the orientation for $\Sig_i(f)$ which agrees with the orientation of $B$ and we take the orientation of $\Sig(f)$ which agrees with the orientation of $\Sig_i(f)$ for $i$ even. As a result of these sign conventions we have the following observation.

\begin{lem}
In the restricted case when the birth-death points of $f$ are framed, then the image under $p^\Sig_\ast$ of the Chern character of $\xi(f)$ is equal to the the alternating sum of images under the push-down operators:
\[
	p_\ast:H^{4\bullet}(\Sig_i(f),\d_0\Sig_i(f))\to H^{4\bullet}(B,\d_0B)
\]
of the Chern character of $\xi_i=\xi|\Sig_i(f)$:
\[
	p^\Sig_\ast(ch(\xi\otimes\CC)=\sum_i (-1)^ip_\ast(ch(\xi_i\otimes\CC)\in H^{4\bullet}(B,\d_0B)
\]
\end{lem}

\begin{thm}[Relative Framing Principle]\label{relative framing principle} Suppose that the manifold $B$ and the stable bundle $\xi=\xi(f)$ are both oriented. Then the higher relative \IK-torsion invariant $\t^\IK(W,M)\in H^{4\bullet}(B,\d_0B)$ is given by the higher torsion of the family of acyclic chain complexes $C(f)$ given by $f$ plus the push down of the normalized Chern character of $\xi$:
\[
	\t^\IK(W,M)=\t(C(f))+p^\Sig_\ast(\normch(\xi))\in H^{4\bullet}(B,\d_0B)
\]
\end{thm}

\begin{proof}
The published version of the Framing Principle \cite{I:ComplexTorsion} assumes that $\d_0B$ is empty. However, the relative case follows easily from the absolute case in the present setting where we have an $h$-cobordism bundle $W$. Just take the base $\d_0W=M$ and embed it into the boundary of a very large dimensional trivial disk bundle $B\times D^N$. Let $\nu_M$ be the vertical normal bundle of $M$ in $B\times S^{N-1}$ and let $\nu_W$ be the extension of $\nu_M$ to $W$. Then we have a new bundle:
\[
	\Delta=B\times D^N\cup D(\nu_W)
\]
over $B$. Since $D(\nu_W)$ is an $h$ cobordism bundle, this is a smooth $N$-disk bundle over $B$ (after rounding off corners). By additivity and invariance after passing to linear disk bundles, we have:
\[
	\t^\IK(W,M)=\t^\IK(D(\nu_W,\nu_M))=\t^\IK(\Delta,B\times D^N)=\t^\IK(\Delta)
\]
But, $\Delta$ is a disk bundle over $B$ which is trivial over $\d_0B$. So, we can collapse $\d_0B$ to a point to get a new bundle $\ov\Delta$ over $B/\d_0B$. The Framing Principle for $\ov\Delta\to B/\d_0B$ is then equivalent to the relative Framing Principle for $(W,M)$.

To do this more precisely, we do the same trick as before, removing a tube $T=D(\nu_M)\times I$ in a collar neighborhood of $B\times S^{N-1}$ and replace it with $W$. The new fiberwise Morse function will be equal to the distance squared from the origin in $B\times D^N-T$ and equal to $f$ (rescaled to match) on $W$. Now we collapse the bundle over $\d_0B$. By construction, the fiberwise generalized Morse function will factor through this quotient bundle and the original Framing Principle applies.
\end{proof}

\begin{proof}[Proof of Theorem \ref{torsion of immersed Hatcher}]
We will start with a fiberwise oriented Morse function on the bundle $E_L^{n,m}(\xi,\eta)\to L$ and then modify it to give a fiberwise oriented generalized Morse function which is framed on the birth-death set.

{The bundle $E_L=E_L^{n,m}(\xi,\eta)$ is obtained from $D_2^n\times D^m(\eta)\times I$ by attaching two handles with cores of dimension $n-1$ and $n$.} (For a more elaborate version of this with more details, see \cite{Goette03}) This means it has a fiberwise Morse function $f:E_L\to I$ which is equal to the projection map to $I$ in a neighborhood of the bottom $D_2^n\times D^m(\eta)\times 0$ and sides $\d(D_2^n\times D^m(\eta))\times I$. Furthermore $f$ will have two critical points over every point $t\in L$. These critical points $x_t,y_t$ have index $n-1$ and $n$ respectively. The vertical tangent bundle of $E_L$ splits as $\e^{n-1}\oplus (\eta\oplus \e^1)$ along the section $x_t$ of $E_L$ where the trivial $n-1$ plane bundle $\e^{n-1}$ is the negative eigenspace of $D^2f_t$ along $x_t$. The vertical tangent bundle of $E_L$ along $y_t$ splits as $\xi\oplus (\eta_0\oplus \e^1)$ where the vector bundle $\xi$, which is homotopically trivial in the sense that $J(\xi)=0$, is the negative eigenspace bundle.

Along $\d_0L$, the bundle $\xi$ is trivial and the handle corresponding to $y_t$ is in cancelling position with the handle corresponding to $x_t$ since they are both standard linear handle along $\d_0L$ by construction. This implies that these critical points can be cancelled along a birth-death set of index $n-1$. Since the negative eigenspace bundle $\xi$ is trivial along this set, this is a framed birth-death set. The new singular set $\Sig(f)$ is now a $q$-manifold with boundary lying over $\d_1L$. It has a framed birth-death set and Morse sets in two indices $\Sig_n(f)$ and $\Sig_{n-1}(f)$. The descending bundles are $\xi_{n-1}=\e^{n-1}$ and $\xi_n=\xi$. These are oriented bundle since they are homotopically trivial. Also the cellular chain complex is trivial at every point. Therefore, by the Framing Principle, the higher relative \IK-torsion of $E_L^{n,m}(\xi,\eta)$ is
\[
	\t^\IK(E_L^{n,m}(\xi,\eta),D^n\times D^m(\eta)\times I)=(-1)^n\normch(\xi)\in H^{4k\bullet}(L,\d_0L)
\]

From this fiberwise oriented generalized Morse function we can construct a fiberwise oriented generalized Morse function $F$ on $E_+^{n,m}(M,\tilde\ll,\xi)=(M-T)\times I\cup E_L$ by taking projection to $I$ on the first piece $(M-T)\times I$ and $f$ on the second piece $E_L$. The singular set of $F$ is the image under $D(\tilde\ll)$ of the singular set of $f$. Consider the following commuting diagram.
\[
\xymatrix{
\Sig_n(f)\ar[r]^\subset\ar[rd]^\simeq&\Sig(f)\ar[d]\ar[r]^(.45){D(\tilde\ll)} &
	\Sig(F)\ar[d]^p\\
&L \ar[r]^\ll& 
	B
	}%end xymatrix
\]
This implies that the image of the push-down of the Chern character of $\xi$ along the map $p$ is equal to the image of the Chern character of $\xi$ under $\ll$. So, by the relative Framing Principle, we have
\[
	\t^\IK(E_+^{n,m}(M,\tilde\ll,\xi),M)=(-1)^np_\ast(\normch(\xi))=(-1)^n\ll_\ast(\normch(\xi))
\]
as claimed.
\end{proof}

% Section:

%\newpage
%%%%%%%%%%%%%%%%%%%%%%%%%%
%
%                Section  {Main Theorem}
%
%%%%%%%%%%%%%%%%%%%%%%%%%%

\section{Main Theorems}

There are two main theorems in this paper. The first concerns the set of possible higher torsion invariants of exotic smooth structures on smooth manifold bundles.

The second theorem is that, rationally stably, the immersed Hatcher construction gives all possible exotic smooth structures on smooth manifold bundles with odd dimensional fibers. This is a combination of the following two theorems. First recall from Section 2 of \cite{Second} that
\[
	\pi_0\std{B}{\d_0}(M)\otimes\RR \cong H_{q-4\bullet}(M,M_{\d_1B})
\]
where the spot $\bullet$ indicates direct sum over all $k>0$ with real coefficients unless otherwise indicated and the image of an exotic smooth structure $M'$ on $M$ is denoted
\[
	\Theta_M(M')=\Theta(M',M)\in H_{q-4\bullet}(M,M_{\d_1B})
\]
and we call it the {\bf (rational) exotic structure class} of $M'$.

\begin{thm}\label{first main theorem}
When the fiber dimension is odd, the rational exotic structure class $\Theta(M',M)$ given by the immersed Hatcher construction $E_+^{n,m}(M,\tilde\ll,\xi)$ is the image of the Poincar\'e dual of twice the normalized Chern character of $\xi$ under the map in homology induced by the embedding $\tilde\ll:(L,\d_1L)\to (M,M_{\d_1B})$. Thus:
\[
	\Theta(M',M)=(-1)^n\tilde\ll_\ast D(2\normch(\xi))
\]
where $\normch(\xi)\in H^{4\bullet}(L,\d_0L)$ is given in Definition \ref {def: normalized Chern character} and $\tilde\ll_\ast\circ D$ is the composition:\[
	H^{4\bullet}(L,\d_0L) \xrightarrow{\cong} H_{q-4\bullet}(L,\d_1L)\xrarrow{\tilde\ll_\ast} H_{q-4\bullet}(M,M_{\d_1B})
\]
\end{thm}

\begin{rem}\label{rem: scalar multiple of integral class} By definition of the normalized Chern character, the exotic structure class $\Theta(M',M)$ lies in the image of
\[
	 H_{q-4\bullet}(M,M_{\d_1B};\zeta(2k+1)\QQ)
\]
In particular, $\Theta(M',M)$ is a scalar multiple of an integral class in every degree.
\end{rem}

\begin{proof} The proof will show the commutativity of the following diagram which is a slightly stronger statement:
\[
\xymatrix{
G(L,\d_0L)\ar[rr]_{top\,E_L^n(-,\eta)}\ar[drr]_{top\,E_+^n(M,\tilde\ll,-)}\ar@/^2pc/[rrrr]_{D\circ (-1)^n2\normch} &&  
\pi_0\std{L}{\d_0}(E)\ar[d]^{D(\tilde\ll)_\ast} \ar[rr]_(.4){D\circ\t^\IK}&&
 H_{q-4\bullet}(L,\d L)\ar[d]^{\ll_\ast}\ar[dl]^{\tilde\ll_\ast}\\
&&  \pi_0\std{B}{\d_0}(M)\ar[r]_(.4){\Theta}&
 H_{q-4\bullet}(M,M_{\d_1B})\ar[r]_{p_\ast} &
 H_{q-4\bullet}(B,\d_1B)
	}%end xymatrix
\]
Here $G(L,\d_0L)=[L/\d_0L,G/O]$. The middle portion can be expanded into the following diagram where $E=D^n\times D^m(\eta)$ is the disk bundle over $L$ which is diffeomorphic to a tubular neighborhood of the image of $\tilde\ll:L\to M$.
\[
\xymatrix{
	\pi_0\std{L}{\d_0}(E)\ar[d]_{D(\tilde\ll)_\ast}\ar[r]_(.45){\Theta_E}\ar@/^2pc/[rr]_{D\circ\t^\IK} &
	%\pi_0\Gamsub{L}{\d_0}\cH^\%_L(E)\ar[d]\ar[r]_(.42){\th_E} &
	H_{q-4\bullet}(E,E_{\d_1})\ar[d]_{D(\tilde\ll)_\ast}\ar[r]_\cong
	& H_{q-4\bullet}(L,\d_1L)\ar[dl]^{\tilde\ll_\ast}\\ 
	\pi_0\std{B}{\d_0}(M) \ar[r]^(.45){\Theta_M}&
	%\pi_0\Gamsub{B}{\d_0}\cH^\%_B(M) \ar[r]^(.42){\th_M}&
	H_{q-4\bullet}(M,M_{\d_1})
	}%end xymatrix
\]

The morphisms $\Theta_E,\Theta_M$ in the second diagram are isomorphisms of vector spaces after tensoring with $\RR$ by Theorem 2.2.3 of \cite{Second} and the vertical maps are all induced by $\tilde\ll:L\to M$ and $D(\tilde\ll):E\to M$. The square commutes by Corollary 2.4.3 of \cite{Second}. The triangle on the right commutes since it comes from a commuting diagram of spaces. The composition of the top two arrows is equal to $D\circ\t^\IK$ by normalization of $\Theta_E$ (Proposition 2.2.4 of \cite{Second}). Therefore, the second diagram commutes. So, the middle quadrilateral in the first diagram commutes.

If we look at the top of the immersed Hatcher handle we get an element
\[
	top(E_+^{n,m}(M,\tilde\ll,\xi))\in \std{B}{\d_0}(M)
\]
which, by construction is the image of the Hatcher disk bundle \[
E'=top(E_+^{n,m}(E,0,\xi))\in \std{L}{\d_0}(E)
\]
under the stratified map $\std{L}{\d_0}(E)\to\std{B}{\d_0}(M)$. The composition of the horizontal mappings on the top row of the first diagram takes $\xi\in G(L,\d_0L)$ to $\t^\IK(E')$. And the last statement we need to prove is:
\[
	\t^\IK(E')=(-1)^n2\normch(\xi).
\]
This follows from the following four equations where $E$ is the bottom of $E_+^{n,m}(E,0,\xi)$. Since $E$ is a linear disk bundle, we have $\t(E)=0$ for any stable torsion invariant $\t$.
\begin{enumerate}
\item $\t(E)=0=\frac12\t(DE)+\frac12\t(\d\vv E)$ by Remark \ref{rem:extension to not closed fibers}.
\item $\t(E')=\frac12\t(DE')+\frac12\t(\d\vv E)$ since $\d\vv E'=\d\vv E$.
\item $\t(\d\vv E_+^{n,m}(E,0,\xi))=\frac12\t(DE)+\frac12\t(DE')$ by Additivity Axiom \ref{defn:axiomatic higher torsion}.
\item $\t^\IK(\d\vv E_+^{n,m}(E,0,\xi))-\t^\IK(DE)=(-1)^n2\normch(\xi)$ by Corollary \ref{cor: torsion of twisted Hatcher sphere bundle}.
\end{enumerate}
This proves the commutativity of the first diagram and the theorem follows.
\end{proof}

\begin{prop}\label{main proposition}
The vector space $H_{q-4\bullet}(M,M_{\d_1B})$ is spanned by the images of the possible maps
\[
 	G(L,\d_0L)\to H_{q-4\bullet}(M,M_{\d_1B})
 \]
given by $\tilde\ll_\ast\circ D\circ(-1)^n2\wt{ch}=\Theta_M\circ top\, E_+^n(M,\tilde\ll,-)$ in the theorem above.
\end{prop}

This proposition is proved below using the Arc de Triomphe construction.

\begin{thm}\label{second main theorem}
When the fiber dimension $N$ of $M\to B$ is odd and $B$ is oriented, the higher \IK-relative torsion of an exotic smooth structure $M'$ on $M$ over $(B,\d_0B)$ and the rational exotic smooth structure class $\Theta(M',M)$ are related by
\[
	D\t^\IK(M',M)=p_\ast\Theta(M',M)
\]
where $D$ is Poincar\'e duality and $p_\ast$ is the map in homology induced by $p:M\to B$. In other words, the following diagram commutes.
\[
\xymatrix{  
\pi_0\std{B}{\d_0}(M)\ar[r]^(.45){\Theta}
\ar@/_2pc/[rr]^{D\circ\t^\IK}
&
 H_{q-4\bullet}(M,M_{\d_1B})\ar[r]^{p_\ast} 
 &
 H_{q-4\bullet}(B,\d_1B)
	}%end xymatrix
\]
\end{thm}

\begin{proof}
The map $p_\ast$ is $\RR$-linear, and Theorem \ref {first main theorem} and Proposition \ref{main proposition} above say that the immersed Hatcher construction gives generators for $\pi_0\std{B}{\d_0}(M)\otimes\RR\cong H_{q-4\bullet}(M,M_{\d_1B})$ and $p_\ast$ sends these generators to their higher relative IK-torsion. The theorem follows.
\end{proof}

We have the following immediate corollary.

\begin{cor}\label{third main theorem}
If $M$ is a smooth bundle over $B$ and both fiber and base are oriented manifolds with odd fiber dimension $N\ge 2q+3$ then the possible values of the higher \IK-relative torsion $\t^\IK(M',M)$ for $M'$ an exotic smooth structure on $M$ which agrees with $M$ over $\d_0B$ will span the image of the push-down map
\[
	p_\ast:H^{N+4\bullet}(M,\d_0M)\to 
		H^{4\bullet}(B,\d_0B)
\]
where $\d_0M=M_{\d_0B}\cup\d\vv M$.
\end{cor}

The immersed Hatcher construction and the Arc de Triomphe do not work to produce exotic smooth structures on smooth bundles $M\to B$ with closed even dimensional fibers. If the vertical boundary $\d\vv M$ is nonempty, these constructions can be used to modify the smooth structure near the vertical boundary and we conjecture that this is the most that can be done. Nevertheless, if there is any way to produce an exotic smooth structure in the case of an even dimensional fiber, it follows easily by reduction to the odd dimensional case that the first part of Theorem \ref{second main theorem} still holds.

\begin{cor}\label{second main theorem: even case}
Theorem \ref{second main theorem} holds for even dimensional fibers. Thus:
\[
	D\t^\IK(M',M)=p_\ast\Theta(M',M)\in H_{q-4\bullet}(B,\d_1B)
\]
\end{cor}

\begin{proof}
If $M,M'\to B$ have even dimensional fibers then $M'\times I,M\times I\to B$ have odd dimensional fibers and we have:
\[
	\t^\IK(M',M)=\t^\IK(M'\times I,M\times I)
\]
since $\t^\IK$ is a stable invariant. Since the fibers of $M\times I\to B$ are odd dimensional, Theorem \ref{second main theorem} applies and
\[
	D\t^\IK(M'\times I,M\times I)=p_\ast\Theta(M'\times I,M\times I)
	\]
This is equal to $p_\ast\Theta(M',M)$ since $\Theta$ is, by definition, a stable invariant.
\end{proof}

% subsection

%\newpage
%-----------------------------------------------------------------------------------
%            sub section {Arc de Triomphe construction}
%-----------------------------------------------------------------------------------

\subsection{Arc de Triomphe 2}\label{subsecA31}

Proposition \ref {main proposition} follows from the Arc de Triomphe construction and the stratified deformation lemma \ref{stratified deformation lemma}. The Arc de Triomphe construction is an extension of the Hatcher construction which rationally stably produces all exotic smooth structures on a compact manifold bundle. The stratified deformation lemma shows that each AdT construction can be deformed into an immersed Hatcher construction.

We explained the basic construction in subsection \ref{ss:AdT}. It only remains to describe the full construction and prove the following theorem.

\begin{thm}[Arc de Triomphe Theorem]\label{AdT lemma}
The AdT construction gives virtually all stable exotic smooth structures on a compact manifold bundle with odd dimensional fibers. In other words, AdT gives all elements in a subgroup of finite index in the group of all stable exotic smooth structures.
\end{thm}

\begin{rem}
If $M\to B$ is a smooth bundle whose fibers are even dimensional, the AdT construction rationally stably produces all exotic smooth structures on $M\times I\to B$. By definition these are stable smooth structures on $M\to B$. So, the theorem implies that \emph{the AdT construction produces virtually all stable smooth structures on all compact manifold bundles}.
\end{rem}

%-----------------------------------------------------------------------------------
%            subsub section {AdT construction}
%-----------------------------------------------------------------------------------

\subsubsection{AdT construction}

The Arc de Triomphe construction goes as follows. Suppose that $M\to B$ is a smooth manifold bundle over a compact oriented $q$-manifold $B$ with odd fiber dimension $N=n+m$ where $m>n>q$. Suppose $\d B=\d_0B\cup \d_1B$ where $\d_0B, \d_1B$ meet along their common boundary. Then we will construct elements of $\utd{B,\d_0B}(M)$, the space of exotic smooth structures on $M$ relative to $\d_0M=M_{\d_0B}\cup \d\vv M$.

\begin{defn}\label{stratified set}
By a {\bf stratified set} \emph{over $B$ with coefficients in $X$} we mean a pair $(\Sig,\psi)$ where $\Sig$ is a compact smooth oriented $q$ manifold together with a smooth mapping $\pi:\Sig\to B$ and $\psi:\Sig\to X$ is a continuous mapping satisfying the following.
\begin{enumerate}
\item $\pi$ sends $\d \Sig$ to $\d B$.
\item $\pi:\Sig\to B$ has only fold singularities, i.e. it is given in local coordinates near critical points by $\pi(x_1,\cdots,x_q)=(x_1^2,x_2,\cdots,x_q)$, and the singular set $\Sig_0$ is a $q-1$ submanifold of $\Sig$ transverse to $\d \Sig$. 
\end{enumerate}
Let $\Sig_+$ and $\Sig_-$ denote the closures of the subsets of $\Sig-\Sig_0$ on which the map $\pi:\Sig\to B$ is orientation preserving and orientation reversing, respectively. Thus $\Sig_-\cap \Sig_+=\Sig_0$ and $\Sig_-\cup \Sig_+=\Sig$.

We say that $(\Sig,\psi)$ is a {\bf stratified subset} of a smooth bundle $M$ over $B$ if $\Sig$ is a smooth submanifold of $M$ and $\pi:\Sig\to B$ is the restriction of $p:M\to B$.
\end{defn}

\begin{rem}\label{rem: vertical vector field}
For any stratified subset $(\Sig,\psi)$ in $M$ over $B$, there is a nowhere zero vertical vector field $v$ along $\Sig_0$ which points from $\Sig_-$ to $\Sig_+$. If the fiber dimension is greater than the base dimension, this vector field extends to a nowhere zero vertical vector field on all of $\Sig$. If the fiber dimension is at least two more than the base dimension, this extension is unique up to homotopy.
\end{rem}

Let $\sd^X_{B,\d_0}(M)$ be the set of stratified deformation classes of stratified subsets $(\Sig,\psi)$ of $M$ over $B$ with coefficients in $X$ so that $\pi(\Sig)$ is disjoint from $\d_0B$. By a {\bf stratified deformation} of stratified subsets $(\Sig,\psi)\simeq(\Sig',\psi')$ of $M$ we mean a stratified subset $(S,\Psi)$ of $M\times I$ over $B\times I$ with coefficients in $X$ so that the image of $S$ in $B\times I$ is disjoint from $\d_0B\times I$ and so that $(\Sig,\psi),(\Sig',\psi')$ are the restrictions of $(S,\Psi)$ to $B\times 0,B\times 1$ respectively. 

Here are two examples that we will use later in the proof of the Stratified Deformation Lemma \ref{stratified deformation lemma}. In both cases, $M=B\times J$ where $J\subset\RR$ is one dimensional. Using Remark \ref{rem: vertical vector field} we will be able to embed these examples into a general stratified subset with sufficiently large fiber dimension.

\begin{eg}[$k$-lens]\label{eg:k-lens} By a \emph{$k$-lens} we mean a stratified set $\Sig$ diffeomorphic to $S^k$ with $\Sig_+$ and $\Sig_-$ both diffeomorphic to $D^k$. Here is an explicit example. Let $M=B\times J$ where $(B,\d_0B)=(D^k,S^{k-1})$ and $J=[0,1]$. Let $\Sig$ be the ellipsoid given by
\[
	\Sig=\{(x,h)\in D^k\times [0,1]\,:\, ||2x||^2+(4h-2)^2=1\}
\] 
with $\Sig_+$ given by $h\ge\frac12$ and $\Sig_-$ given by $h\le\frac12$. This set can also be given in polar coordinates by the equation
\[
	||2r||^2+(4h-2)^2=1
\]
where $(r,\th,h)\in [0,1]\times S^{k-1}\times J$. Since $\th\in S^{k-1}$ does not occur in the equation, the set $\Sig$ is given by \emph{spinning} the subset of the $r,h$-plane given be the above equation.
\begin{figure}[htbp]
\begin{center}
%
%\vs5
{
\setlength{\unitlength}{2cm}
%\centerline
{\mbox{
\begin{picture}(2,1)
      \thicklines
      \qbezier(1,.25)(2,.25)(2,.5)
    \thinlines
      \qbezier(1,.75)(2,.75)(2,.5)
      \put(2,0){
      \thicklines
      \qbezier(-1,.25)(-2,.25)(-2,.5)
    \thinlines
      \qbezier(-1,.75)(-2,.75)(-2,.5)
      }
      \put(2,.2){$\Sigma_-$}
      \put(2,.7){$\Sigma_+$}
%     \put(1,0){ \line(0,1){1}}
     %
     \put(1,0){
     \line(0,1){1} % axis
     %	Rotation arrow at (0,0)
      \qbezier(-.2,0)(-.2,-.1)(0,-.1)  
     \qbezier(.2,0)(.2,-.1)(0,-.1)  
     \qbezier(.2,0)(.19,-.05)(.22,-.06)
     \qbezier(.2,0)(.19,-.05)(.14,-.03)
     \put(.3,-.1){$S^{k-1}$}
   }
\end{picture}}
}}
%\vs5
%\caption{default}
\label{fig:k-lens}
\end{center}
\end{figure}
\end{eg}

\begin{eg}[mushroom]\label{eg:mushroom}
Let $M=B\times J$ where $(B,\d_0B)=(D^{k},\emptyset)$ and $J=[-3,3]$. Let $\Sigma\subset M\times (-1,1)$ be the stratified deformation given in polar coordinates $(r,\th,h,t)\in [0,1]\times S^{k+1}\times J\times (-1,1)$ by the equation
\[
	4(r^2+t^2)=h-\frac{h^3}3+1
\]
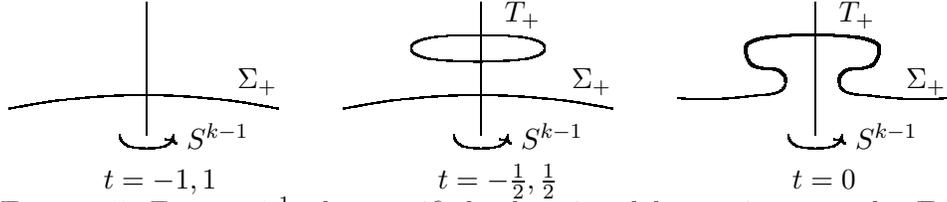
\begin{figure}[htbp]
\begin{center}
%
%\vs5
{
\setlength{\unitlength}{.7in}
%\centerline
{\mbox{
\begin{picture}(7,1.2)
\put(0,.2){ % begin A
%      \thicklines
    \thinlines
      \qbezier(1,.3)(1.5,.3)(2,.2)
      \put(2,0){
%      \thicklines
            \qbezier(-1,.3)(-1.5,.3)(-2,.2)
      }
      \put(1.7,.35){$\Sigma_+$}
%      \put(1.2,.85){$T_+$}
%     \put(1,0){ \line(0,1){1}}
     %
     \put(.95,0){ % begin rotation
     \line(0,1){1} % axis
     %	Rotation arrow at (0,0)
      \qbezier(-.2,0)(-.2,-.1)(0,-.1)  
     \qbezier(.2,0)(.2,-.1)(0,-.1)  
     \qbezier(.2,0)(.19,-.05)(.22,-.06)
     \qbezier(.2,0)(.19,-.05)(.14,-.03)
     \put(.3,-.1){$S^{k-1}$}
   } % end rotation 
   \put(.7,-.4){$t=-1,1$}
   } %end A
\put(2.5,.2){ % begin B
%      \thicklines
    \thinlines
      \qbezier(1,.75)(1.5,.75)(1.5,.65)
      \qbezier(1,.55)(1.5,.55)(1.5,.65)
      \qbezier(1,.3)(1.5,.3)(2,.2)
      \put(2,0){
%      \thicklines
      \qbezier(-1,.75)(-1.5,.75)(-1.5,.65)
      \qbezier(-1,.55)(-1.5,.55)(-1.5,.65)
            \qbezier(-1,.3)(-1.5,.3)(-2,.2)
      }
      \put(1.7,.35){$\Sigma_+$}
      \put(1.2,.85){$T_+$}
%     \put(1,0){ \line(0,1){1}}
     %
     \put(.95,0){ % begin rotation
     \line(0,1){1} % axis
     %	Rotation arrow at (0,0)
      \qbezier(-.2,0)(-.2,-.1)(0,-.1)  
     \qbezier(.2,0)(.2,-.1)(0,-.1)  
     \qbezier(.2,0)(.19,-.05)(.22,-.06)
     \qbezier(.2,0)(.19,-.05)(.14,-.03)
     \put(.3,-.1){$S^{k-1}$}
   } % end rotation 
   \put(.7,-.4){$t=-\frac12,\frac12$}
   } %end B
   \put(5,.2){ % begin C
      \thicklines
      \qbezier(1,.75)(1.5,.75)(1.5,.65)
      \qbezier(1.35,.5)(1.5,.5)(1.5,.65)
      \qbezier(1.35,.5)(1.2,.5)(1.2,.4)      
    \thinlines
      \qbezier(1.4,.3)(1.8,.25)(2,.28)
      \qbezier(1.4,.3)(1.2,.3)(1.2,.4)      
      \put(2,0){
      \thicklines
      \qbezier(-1,.75)(-1.5,.75)(-1.5,.65)
      \qbezier(-1.35,.5)(-1.5,.5)(-1.5,.65)
      \qbezier(-1.35,.5)(-1.2,.5)(-1.2,.4)      
     \thinlines
      \qbezier(-1.4,.3)(-1.8,.25)(-2,.28)
     \qbezier(-1.4,.3)(-1.2,.3)(-1.2,.4)      
      }
      \put(1.7,.35){$\Sigma_+$}
      \put(1.2,.85){$T_+$}
%     \put(1,0){ \line(0,1){1}}
     %
     \put(.95,0){ % begin rotation
     \line(0,1){1} % axis
     %	Rotation arrow at (0,0)
      \qbezier(-.2,0)(-.2,-.1)(0,-.1)  
     \qbezier(.2,0)(.2,-.1)(0,-.1)  
     \qbezier(.2,0)(.19,-.05)(.22,-.06)
     \qbezier(.2,0)(.19,-.05)(.14,-.03)
     \put(.3,-.1){$S^{k-1}$}
   } % end rotation 
   \put(.85,-.4){$t=0$}
   } %end C
\end{picture}}
}}
%\vs5
\caption{For $t=\pm\frac12$, the stratified subset is a $k$-lens union a regular $\Sig_+$ component. For $t=\pm1$, this $k$-lens disappears. For $t=0$, the rotated shape resembles a \emph{mushroom}. We call the deformation $t:-1\to 0$ ``planting a mushroom.''
}
\label{fig: mushroom}
\end{center}
\end{figure}
\end{eg}

%Note that, in Example \ref{eg:mushroom}, the mushroom is symmetric about vertical axis (given by $r=0$). This allowed us to rotate and obtain the smooth stratified deformation which ``merges'' the ``S'' shape and its mirror image. This is an example of the following.

A stratified subset is $\Sig\subset M$ together with a continuous mapping $\psi:\Sig\to X$. The coefficient spaces $X$ that we are interested in are $X=BSO$, classifying oriented stable vector bundles over $\Sig$ and $X=G/O=SG/SO$ classifying vector bundles with homotopy trivializations of the corresponding spherical fibration. The latter is the input for Hatcher's construction and the Arc de Triomphe construction will be a mapping
\[
	AdT:\sdgo{B}{\d_0}(M)\to \std{B}{\d_0}(M)
\]
The claim is that this map is rationally split surjective. In other words, rationally stably, all exotic tangential smoothings on $M$ are given by the construction that we will now give.

The idea of the construction is to attach negative Hatcher handles along $\Sig_-$ and positive Hatcher handles along $\Sig_+$ and have them cancel along $\Sig_0$. The map $\psi:\Sig\to G/O$ gives the bundle $\xi$ in the Hatcher handle.

Suppose that $m>n>q$ and $M\to B$ is a smooth bundle with fiber dimension $m+n$ which we assume is odd ($2q+3$ is the minimum). Suppose we have a stratified subset $\Sig\subset M$ with coefficient map $\psi:\Sig\to G/O$. This gives a stable vector bundle $\xi$ over $\Sig$. Let $\eta$ be the unique $m$-plane bundle over $\Sig$ isomorphic to the pull-back of the vertical tangent bundle of $M$ and let $\eta_-,\eta_+,\eta_0$ be the restrictions of $\eta$ to $\Sig_-,\Sig_+,\eta_0$. Then we have an embedding
\[
	D(\tilde\pi_+):D^n\times D^m(\eta_+)\into M
\]
lying over the restriction $\pi_+:\Sig_+\to B$ of $\pi$ to $\Sig_+$. This gives a tubular neighborhood of $\Sig_+$. Replacing $+$ with $-$ we get $D(\tilde\pi_-)$ lying over $\pi_-$ giving a thickening of $\Sig_-$. The embeddings $D(\tilde\pi_+)$ and $D(\tilde\pi_-)$ are disjoint except near $\Sig_0$. To correct this we move $D(\tilde\pi_-)$ slightly in the fiber direction near $\Sig_0$ so that the images of $D(\tilde\pi_+)$ and $D(\tilde\pi_-)$ are disjoint everywhere. We do this move systematically by moving in the direction of, say, the last coordinate vector $e_n$ in $D^n$. The result will be that the image of $D(\tilde\pi_-)$ will no longer contain $\Sig_-$ close to $\Sig_0$.

Do this in such a way that there is an embedding
\[
	D(\tilde\pi_0):D^n\times D(\eta_0)\to M
\]
so that $D(\tilde\pi_-)(x,y)=D(\tilde\pi_0)(\tfrac14(x+2e_n),y)$ and $D(\tilde\pi_+)(x,y)=D(\tilde\pi_0)(\tfrac14(x-2 e_n),y)$. Or, start with embedding $D(\tilde\pi_0)$ and move the mappings $D(\tilde\pi_+),D(\tilde\pi_-)$ vertically (along the fibers) so that they land in the two halves of the image of $D(\tilde\pi_0)$ as indicated.

Take the bundle $M\times I$ over $B$ and, using the map $D(\tilde\pi_+)$ we attach the positive Hatcher handle $B^{n,m}(\xi,\eta_+)$ along its base $\d_0B^{n,m}(\xi,\eta_+)=D^n\times D^m(\eta_+)\times0$ to the top $M\times 1$ of $M\times I$. Then we attach the negative Hatcher handle $A^{n,m}(\xi,\eta_-)$ to the top of $M\times I$ using the composite map
\[
	E^{n,m}(\xi,\eta_-)\xrarrow{F(j)}D^n\times D^m(\eta_-)\xrarrow{D(\tilde\pi_-)}M
\]
Since the images of $D(\tilde\pi_+)$ and $D(\tilde\pi_-)$ are disjoint, these attachments are disjoint.

Over $\pi(\Sig_0)$ we have a positive and negative Hatcher handle attached on the interior of the image of $D(\tilde\pi_0)$. Next, we slide the attachment map for the negative Hatcher handle until it ``cancels'' the positive Hatcher handle. It is very easy to see how this works. Over $\Sig_0$ the negative Hatcher handle $A^{n,m}(\xi,\eta_0)$ is attached along its base $\d_0A^{n,m}(\xi,\eta_0)=E^{n,m}(\xi,\eta_0)$ and the positive Hatcher handle is 
\[
	B^{n,m}(\xi,\eta_0)=D^n\times D^m(\eta_0)\cup_{h\times1}E^{n,m}(\xi,\eta_0)\times [1,2]
\]
By Lemma \ref{third trivial lemma}, we can slide the base $E^{n,m}(\xi,\eta_0)$ of $A^{n,m}(\xi,\eta_0)$ along the top of the $M\times 1\cup B^{n,m}(\xi,\eta_+)$ until it is equal to $E^{n,m}(\xi,\eta_0)\times2\subseteq  B^{n,m}(\xi,\eta_0)$. We can do this in a precise way since we are working inside of the model which is the image of $D(\tilde\pi_0)$ in $M\times 1$. We extend this deformation (arbitrarily) to $A^{n,m}(\xi,\eta_-)$. Then we will have the desired bundle over $B$ whose fibers are $h$-bordisms with base equal to the original bundle $M$. We call this new bundle $W(\Sig,\psi)$ (suppressing $n,m$):
\[
	W(\Sig,\psi)=M\times I\cup B^{n,m}(\xi,\eta_+)\cup A^{n,m}(\xi,\eta_-)
\]
To be sure, we need to round off the corners. And we also need to taper off the cancelling Hatcher handles along $\Sig_0$. But, along $\Sig_0$, the two Hatcher handles cancel and we have a local diffeomorphism of $W(\Sig,\psi)$ with $M\times I$ near $\Sig_0$. Using this diffeomorphism we can identify $W$ with $M\times I$ along this set and we have a smooth bundle over $B$. The local diffeomorphism exists by Proposition \ref {second AdT cancellation lemma}. The reason that we have a bundle at the end is because, in a neighborhood of the AdT construction along $\Sig_0$ we either have two Hatcher handles, which are a smooth continuation of what we have at $\Sig_0$ or we have $M\times I$ locally (which means we are only looking at the portion in the image of $D(\tilde\pi_0)$) and there we are using the diffeomorphism given by Proposition \ref {second AdT cancellation lemma} to identify $M\times I$ with the $M\times I$ with the pair of Hatcher handles attached. So, we have local triviality and thus a smooth bundle $W\to B$. Let
\[
	AdT(\Sig,\psi)=top(W(\Sig,\psi))
\]
with tangential homeomorphism given by $W$.
If we have any deformation of $(\Sig,\psi)$ then we can apply the same construction to this stratified set over $B\times I$ and we get a isotopy between the two constructions showing that $AdT(\Sig,\psi)$ changes by an isotopy. %Thus we get the following lemma.

\begin{prop}\label{sd is a group}
(a) The AdT construction as described above gives a well defined mapping
\[
	AdT:\sdgo{B}{\d_0}(M)\to \pi_0\std{B}{\d_0}(M)
\]
from the set of stratified deformation classes of stratified subsets $(\Sig,\psi)$ of $M$ with coefficients in $G/O$ to the space of stable tangential smoothings of $M$. 

(b) This mapping is a homomorphism of additive groups where addition in $\sdgo{B}{\d_0}(M)$ is given by disjoint union and addition in $ \pi_0\std{B}{\d_0}(M)$ is given by the little cubes operad on the stabilization.
\end{prop}

\begin{proof} It is clear that $\sdgo{B}{\d_0}(M)$ is a monoid with addition given by disjoint union using transitivity to make any two stratified subsets of $M$ disjoint by a small perturbation. We also have additive inverses given as follows.

For any stratified subset $(\Sig,\psi)$ in $M$, we claim that there are  stratified subsets $(S,\Psi), (T,\Psi)$ each deformable to the empty set by a stratified deformation, making them equal to zero in the group $\sdgo{B}{\d_0}(M)$, and so that $(S\cup T,\Psi)$ is also deformable into the disjoint union of $(\Sig,\psi)$ and another stratified subset $(U,\psi')$ of $M$. This makes $(U,\psi')$ the additive inverse of $(\Sig,\psi)$. The construction of $(S,\Psi), (T,\Psi)$ is as follows.

By Remark \ref{rem: vertical vector field}, there is a nowhere zero vertical vector field $v$ along $\Sig$ so that, along $\Sig_0$, it point from $\Sig_-$ to $\Sig_+$. Using this vector field, we can embed a ``ribbon'' $R_+\cong\Sig_+\times I$ in $M$ so that $R_+$ contains $\Sig_+$ as $\Sig_+\times 0$. In this ribbon we take $S=S_+\cup S_-$ where $S_+=\Sig_+\times \frac13$ and $S_-=\Sig_+\times\frac23$. As we approach $\Sig_0$ we replace $\frac13,\frac23$ by numbers converging to $\frac12$. Let $\Psi:S\to X$ be given by $\Psi(x,t)=\psi(x)$ for all $(x,t)\in \Sig_+\times I$. \vs2

\ul{Claim 1}. $(S,\Psi)$ is deformable into the empty set.

Pf: Deform $S$ inside the ribbon by letting the coordinates $\frac13,\frac23$ converge to $\frac12$. Extend $\Psi$ using the same equation. This gives the null deformation.\vs2

Construct $(T,\Psi)$ inside of the ribbon $R_-\cong \Sig_-\times I$ containing $\Sig_-$ as $\Sig_-\times 1$ in a similar way.\vs2

\ul{Claim 2}. $(S\cup T,\Psi)$ is deformable into the disjoint union of $(\Sig,\psi)$ and another stratified subset of $M$.

Pf: Along $\Sig_0$ we can merge the bottom of $S$ with the top of $T$, just like the deformation $t=\frac12$ to $t=0$ in Example \ref{eg:mushroom} above. This is illustrated in the following diagram.
\begin{figure}[htbp]
\begin{center}
%
%\vs5
{
\setlength{\unitlength}{1in}
%\centerline
{\mbox{
\begin{picture}(4,1)
      \thicklines
%    \thinlines
%
\put(0,0.3){ %begin A
\qbezier(.1,.5)(1,1.3)(1.8,.5)
\qbezier(.1,.5)(.05,.45)(0.1,.4)
\qbezier(0.1,.4)(.13,.38)(.2,.45)
\qbezier(.1,.5)(.05,.45)(0.1,.4)
\qbezier(1.8,.4)(1.77,.38)(1.7,.45)
\qbezier(1.8,.5)(1.85,.45)(1.8,.4)
\qbezier(.2,.45)(1,1.15)(1.7,.45)
\put(.3,.8){$S_-$}
\put(.6,.6){$S_+$}
}
\put(0,1){
\qbezier(.1,-.5)(1,-1.3)(1.8,-.5)
\qbezier(.1,-.5)(.05,-.45)(0.1,-.4)
\qbezier(0.1,-.4)(.13,-.38)(.2,-.45)
\qbezier(.1,-.5)(.05,-.45)(0.1,-.4)
\qbezier(1.8,-.4)(1.77,-.38)(1.7,-.45)
\qbezier(1.8,-.5)(1.85,-.45)(1.8,-.4)
\qbezier(.2,-.45)(1,-1.15)(1.7,-.45)
\put(.3,-.9){$T_+$}
\put(.6,-.65){$T_-$}
} % end A
\put(2,.6){$\Rightarrow$}
\put(2.3,0){ % begin B
\put(0,0.3){
\qbezier(.1,.5)(1,1.3)(1.8,.5)
\qbezier(.1,.5)(.05,.45)(0.1,.4)
%\qbezier(0.1,.4)(.13,.38)(.2,.45)
\qbezier(0.2,.25)(.13,.35)(.2,.45) % new 1
\qbezier(1.7,.25)(1.77,.35)(1.7,.45) % new 2
\qbezier(0.1,.4)(.13,.35)(.1,.3) % new 3
\qbezier(1.8,.4)(1.77,.35)(1.8,.3) % new 4
\qbezier(.1,.5)(.05,.45)(0.1,.4)
%\qbezier(1.8,.4)(1.77,.38)(1.7,.45)
\qbezier(1.8,.5)(1.85,.45)(1.8,.4)
\qbezier(.2,.45)(1,1.15)(1.7,.45)
%\put(.3,.8){$S_-$}
\put(.6,.6){$\Sig_+$}
}
\put(0,1){
\qbezier(.1,-.5)(1,-1.3)(1.8,-.5)
\qbezier(.1,-.5)(.05,-.45)(0.1,-.4)
%\qbezier(0.1,-.4)(.13,-.38)(.2,-.45)
\qbezier(.1,-.5)(.05,-.45)(0.1,-.4)
%\qbezier(1.8,-.4)(1.77,-.38)(1.7,-.45)
\qbezier(1.8,-.5)(1.85,-.45)(1.8,-.4)
\qbezier(.2,-.45)(1,-1.15)(1.7,-.45)
%\put(.3,-.9){$T_+$}
\put(.6,-.65){$\Sig_-$}
}
}% end B
\end{picture}}
}}
%\vs5
\caption{The sum of the trivial stratified sets $(S,\Psi),(T,\Psi)$ deform to the disjoint union of $(\Sig,\psi)$ and another stratified set.}
\label{figure eight}
\end{center}
\end{figure}
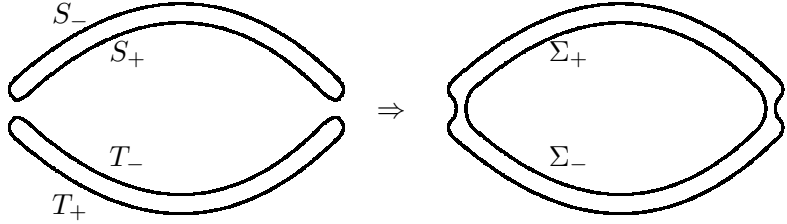

To show that the mapping $AdT$ is additive, we take two smooth structures $\th_1,\th_2$ on the stabilized $M\times D^{2k-1}\times I$ which by the stabilization construction are equal to the original smooth structure on $\d\vv (M\times D^{2k-1})\times I\cup M\times D^{2k-1}\times 0$ and on the complements of $E_1\times D^{2k}$ and $E_2\times D^{2k}$ respectively. By transversality, these two subsets, the supports of the two exotic smooth structures are disjoint. Therefore, by Proposition 1.5.10 of \cite{Second}, $\th_1+\th_2$ is given by changing the smooth structure of both $E_1$ and $E_2$. This show that $SdT$ is additive.
\end{proof}

\begin{rem}\label{formula for inverse of stratified set} 
The proof above shows that the inverse of $(\Sig,\psi)\in \sdgo{B}{\d_0}(M)$ has the form $(\Sig',\psi')$ where $\psi'$ is the composition 
\[
\Sig'\xrarrow \r\Sig\xrarrow \psi G/O
\]
Where $\r:\Sig'\to \Sig$ maps a subset $U_+\subset\Sig_+'$ homeomorphically onto the interior of $\Sig_-$ and a subset $U_-\subset\Sig_-'$ homeomorphically onto the interior of $\Sig_+$. Furthermore, the restriction of $\r$ to $U_+\cup U_-$ is compatible with the projection to $B$.
\end{rem}

\begin{prop}\label{prop: induced map on ovAdT}
If $\psi:\Sig\to G/O$ is trivial then so is $AdT(\Sig,\psi)$. Therefore, $AdT$ induces a homomorphism
\[
	\ov{AdT}:\ovsdgo{B}{\d_0}(M)\to \pi_0\std{B}{\d_0}(M)
\]
Where $\ovsdgo{B}{\d_0}(M)$ is the quotient of $\sdgo{B}{\d_0}(M)$ by all $(\Sig,\psi)$ where $\psi$ is null homotopic.
\end{prop}

\begin{proof}
If $\psi$ is constant then the positive and negative Hatcher handles in the Arc de Triomphe construction are standard disk bundles and attaching these to the top of $M\times I$ will not change its fiber diffeomorphism type.
\end{proof}

%-----------------------------------------------------------------------------------
%            subsub section {homotopy calculation}
%-----------------------------------------------------------------------------------

\subsubsection{Homotopy calculation}\label{ss312}

To prove Theorem \ref{AdT lemma} we need calculations in the form of more commuting diagrams.

Let
\[
	D\normch:\sdgo{B}{\d_0}(M)\to H_{q-4\bullet}(M,\d_1M)
\]
be the mapping given by sending $(\Sig,\psi)$ to the image of the normalized Chern character of the bundle $\xi$ under the mapping
\[
	\normch(\xi)\in  H^{4\bullet}(\Sig)\cong H_{q-4\bullet}(\Sig,\d\Sig)\xrarrow{j_\ast}  H_{q-4\bullet}(M,\d_1M)
\]
induced by the inclusion $j:(\Sig,\d\Sig)\to (M,\d_1M)$. Since $\xi$ is an oriented bundle, the Framing Principle applies to prove the following.

\begin{lem}\label{IK torsion of AdT3}
The following diagram commutes if $n+m$ is odd.
\[
\xymatrix{
\ovsdgo{B}{\d_0}(M)\ar[rr]_{\ov{AdT}}
\ar@/_2pc/[rrr]_{(-1)^n2D\normch}
&&  
\pi_0\std{B}{\d_0}(M)%\ar[r]_(.4){\Theta}
\ar@/_2pc/[rr]_{D\circ\t^\IK}
&
 H_{q-4\bullet}(M,M_{\d_1B})\ar[r]_{p_\ast} 
 &
 H_{q-4\bullet}(B,\d_1B)
	}%end xymatrix
\]
\end{lem}

Although we claim that the Framing Principle implies this lemma, we don't need to verify it since this lemma follows from the next lemma.

\begin{lem}\label{representation by immersed Hatcher}
Every element of $\ovsdgo{B}{\d_0}(M)$ is in the image of a homomorphism
\[
	\Sig_{\tilde\ll}:G(L,\d_0L)\to \ovsdgo{B}{\d_0}(M)
\]
where $\ll:(L,\d_1L)\to (B,\d_1B)$ is a codimension $0$ immersion covered by an embedding $\tilde\ll:L\to M$ which makes the following diagram commute.
\[
\xymatrix{
G(L,\d_0L)\ar[d]_{\Sig_{\tilde\ll}}
%\ar[rr]_{top\,E_L^n(-,\eta)}
\ar[drr]^{top\,E_+^n(M,\tilde\ll,-)}
\ar[rrr]^{D\circ(-1)^n2\normch} 
&&  
%\pi_0\std{L}{\d_0}(E)\ar[d]^{D(\tilde\ll)_\ast} %\ar[r]_(.4){\Theta}
%\ar@/^2pc/[rr]_{\t^\IK}
&
% H_{q-4\bullet}(E,E_{\d_1})\ar[d]^{D(\tilde\ll)_\ast}\ar[r]_\cong
 H_{q-4\bullet}(L,\d_1 L)%\ar[d]^{\ll_\ast}
 \ar[d]^{\Sig_{\tilde\ll}}
 \\
\ovsdgo{B}{\d_0}(M)\ar[rr]_{\ov{AdT}}
\ar@/_2pc/[rrr]_{(-1)^n2D\normch}
&&  
\pi_0\std{B}{\d_0}(M)%\ar[r]_(.4){\Theta}
%\ar@/_2pc/[rr]_{\t^\IK}
&
 H_{q-4\bullet}(M,M_{\d_1B})%\ar[r]_{p_\ast} 
 &
 %H_{q-4\bullet}(B,\d_1B)
	}%end xymatrix
\]
\end{lem}

\begin{proof}[Proof of Lemma \ref {IK torsion of AdT3}]
First we note that both maps coming out of $\sdgo{B}{\d_0}(M)$ factor through $\ovsdgo{B}{\d_0}(M)$. Each element then lifts to $G(L,\d_0L)$. Next we chase the diagram at the beginning of the proof of Theorem \ref{first main theorem} to show that the two images of this element in $\bigoplus H_{q-4k}(B,\d_1)$ are equal. The diagram in Lemma \ref{representation by immersed Hatcher}  above shows that the two images obtained are the same as the two images in the diagram of Lemma \ref {IK torsion of AdT3} which we are proving.
\end{proof}

\begin{proof}[Proof of Lemma \ref{representation by immersed Hatcher}]
The mapping $\Sig_{\tilde\ll}$ takes a map $\xi:L\to G/O$ which is trivial over $\d_0L$ and produces a stratified subset 
\[
	\Sig_{\tilde\ll}(\xi)=(\Sig,\psi)
\]
where $\Sig$ is two copies of $L$, thus $\Sig_-\cong\Sig_+\cong L$, glued together along $\d_0L$ and embedded in $M$ using two small perturbations of the embedding $\tilde\ll:L\to M$. The mapping psi is equal to $\xi$ on $\Sig_+$ and is trivial on $\Sig_-$. Since $\psi$ is trivial on $\Sig_-$, the negative Hatcher handles in $W(\Sig,\psi)$ are standard disk bundles. So, the bundle $AdT(\Sig,\psi)$ will not change if we remove these ``trivial'' Hatcher handles. The result is then equivalent to the immersed Hatcher handle. This shows that the triangle in the diagram commutes. Commutativity of the (curved) square follows from the definition of $D\normch(\xi)$ on $\sdgo{B}{\d_0B}$, namely, $D\normch\Sig_{\tilde\ll}(\xi)$ is the push-forward along the embedding $D(\tilde\ll):E\to M$ of the Poincar\'e dual of the normalized Chern character of $\xi$ as a bundle over $L$.

It remains to prove the element-wise surjectivity statement. This follows from the stratified deformation lemma \ref{stratified deformation lemma} whose proof we leave until the end. This lemma shows that any stratified subset $(\Sig,\psi)$ of $M$ can be deformed so that every component of $\Sig_-$ is contained in a disjoint contractible subset of $\Sig$. Then we can deform $\psi$ so that it is constant on each component of $\Sig_-$ and therefore also on $\Sig_0$. Then let $(L,\d_0L)=(\Sig_+,\Sig_0)$ and let $\ll:L\to B$ be the map $\pi_+:\Sig_+\to B$. Let $\tilde\ll:L\to M$ be the inclusion map of $\Sig_+$. Then we claim that the image of $(\Sig,\psi)$ in $\ovsdgo{B}{\d_0}(M)$ is equal to the image $\Sig_{\tilde\ll}(\xi_+)$ of $\xi_+=\xi|\Sig_+\in G(L,\d_0L)$. Since we started with an arbitrary element of $\sdgo{B}{\d_0}(M)$ this will prove the lemma.

To see that $(\Sig,\psi)$ and $\Sig_{\tilde\ll}(\xi_+)$ are equal in $\ovsdgo{B}{\d_0}(M)$, we just take the difference $\Sig_{\tilde\ll}(\xi_+)-(\Sig,\psi)$. The negative of $(\Sig,\psi)$ given in Remark \ref{formula for inverse of stratified set} has the form $(\Sig',\psi')$ where $\psi'=\psi\circ\r:\Sig'\to\Sig\to G/O$. But then $\psi'$ is trivial on $\r^{-1}(\Sig\delete U_-)$ and $U_-\subset \Sig'_-\cong \Sig_+$ has the same $G/O$ coefficient map as $\Sig_{\tilde\ll}(\xi_+)$ has on its positive part. Therefore, the subset $U_-$ of the negative part of $\Sig'$ cancels the interior of the positive part of $\Sig_{\tilde\ll}(\xi_+)$ by a stratified deformation. The result has trivial coefficient map to $G/O$ and therefore is trivial in $\ovsdgo{B}{\d_0}(M)$ as claimed.
\end{proof}

%-----------------------------------------------------------------------------------
%            subsub section {proof of the AdT Theorem}
%-----------------------------------------------------------------------------------

\subsubsection{Proof of the AdT Theorem}

The Arc de Triomphe Theorem \ref{AdT lemma} will follow from the following first version of the theorem.

\begin{lem}\label{first version of AdT Lemma}
The mapping
\[
	D\normch:\sdgo{B}{\d_0}(M)\to
	H_{q-4\bullet}(M,\d_1M)
\]
is rationally surjective in the sense that its image generates $H_{q-4\bullet}(M,\d_1M)$ as a vector space over $\RR$.
\end{lem}

\begin{proof} We review the properties of generalized Morse functions (GMF) as described in \cite{I:GMF}, namely the singular set of a fiberwise GMF is a stratified set $\Sig$ together with a coefficient mapping $\Sig\to BO$. See also \cite{Goette08} for the relationship between generalized Morse functions and analytic torsion.

{Consider the bundle $M\times I\to B$ and consider an arbitrary fiberwise generalized Morse function} $f:M\times I\to I$ which agrees with the projection map over $\d_0B$ and in a neighborhood of the vertical boundary. Thus $f=pr_I$ on the set
\[
 A=\d_0M\times I\cup M\times \{0,1\}
\]
The fact that $f$ is a fiberwise GMF is equivalent to the property that the $f$ is in general position, so that its singular set $\Sig(f)$ is a submanifold of $M\times I$, and so that the projection map $\Sig(f)\to B$ has only fold singularities and the Morse point set which are the regular points of the projection $\Sig(f)\to B$ are stratified by index $i$. We will use just the sign $(-1)^i$ making $\Sig_+$ into the set of Morse points of even index and $\Sig_-$ the set of odd index Morse points of $f$. It is important to note that $\Sig(f)$ is a manifold with boundary and $\d\Sig(f)=\Sig(f)\cap M_{\d_1B}\times I$.

The singular set is the inverse image of zero under the vertical derivative $D\vv(f)$ of $f$ and therefore a framed manifold with boundary. (Add the vertical normal bundle to see the framing.) Since the space of all smooth functions on $M\times I$ equal to $pr_I$ on $A$ is contractible and contains a function without critical points, this framed manifold is framed null cobordant and represents the trivial element of the fiberwise framed cobordism group of $M$ relative to $M_{\d_1B}$ which is $\pi_0\Gamsub{B}{\d_0} Q_B(M)$ where $Q_B(M)$ is the bundle over $B$ with fiber $Q(X_+)=\Omega^\infty\Sig^\infty(X_+)$ over $b\in B$ if $X$ is the fiber of $M\times I$ over $b$.

The negative eigenspace of $D^2(f)$ gives a stable vector bundle $\xi$ over $\Sig(f)$. So $\Sig(f)$, together with $\xi$ gives a stratified subset of $M\times I$ with coefficients in $BO=\colim BO(k)$. Since $\Sig(f)$ is a framed manifold with boundary which is framed null cobordant when we ignore this vector bundle, we get an element of the kernel of the map from the fiberwise framed cobordism group of $BO\times M$ to that of $M$. This kernel is $\pi_0$ of the fiber of the map:
\[
	\g:\Gamsub{B}{\d_0}Q_B(BO\times M)\to \Gamsub{B}{\d_0}Q_B(M)
\]
In \cite{I:GMF}, it is shown that the space of generalized Morse functions on a manifold $X$ is $\dim X$-equivalent to $Q(BO\wedge X_+)$. If we apply that theorem fiberwise, we get that the space of fiberwise generalized Morse functions on $M\times I$ has the $n+m-q$ homotopy type of the fiber of the map $\g$ above.

However, it is a standard homotopy argument to show that there is a split surjection
\[
	Q(BO\wedge X_+)\to \Omega^\infty(BO\wedge X_+)
\]
which is rationally equivalent to the homology of $X$ in every 4th degree since $BO$ is rationally equivalent to $\prod_{k>0}K(\ZZ,4k)$. Therefore, $\pi_0(fiber(\g))$ has a split summand which is rationally isomorphic to the group:
\[
	H:=H_{q-4\bullet}(M,M_{\d_1B};\QQ)
\]
by the basic homotopy calculation (Corollary 2.2.2 of \cite{Second}).

This implies that a set of generators for the vector space $H\otimes \RR$ is given by taking $D\normch(\Sig,\xi)$ for all possible stratified sets $(\Sig,\xi)\in \sd^{BO}_{B,\d_0}(M\times I)$ given by all fiberwise generalized Morse functions on $M\times I$ fixing the subspace $A$. Using the fact that the group $J(\Sig)$ is finite with order, say $m$, we know that $J(\xi^m)=0$ in $J(\Sig)$ and therefore lifts to a map $\Sig\to G/O$. So, these various stratified sets $(\Sig,\xi^m)\in \sdgo{B}{\d_0}(M\times I)$ will have $D\normch(\Sig,\xi^m)$ generating the vector space $H\otimes\RR$ as claimed.
\end{proof}

\begin{lem}\label{Th AdT=2ch}
The following diagram commutes 
\[
\xymatrix{
\sdgo{B}{\d_0}(M)\ar[rr]_{{AdT}}
\ar@/_2pc/[rrr]_{(-1)^n2D\normch}
&&  
\pi_0\std{B}{\d_0}(M)\ar[r]_(.4){\Theta}
%\ar@/_2pc/[rr]_{\t^\IK}
&
 H_{q-4\bullet}(M,M_{\d_1B})%\ar[r]_{p_\ast} 
% &
% H_{q-4\bullet}(B,\d_1B)
 } % end xymatrix
\]where $\Theta:M'\mapsto \Theta(M',M)$ gives the rational exotic structure class of $M'$.
\end{lem}

This lemma proves the Arc de Triomphe Theorem \ref{AdT lemma} since we just proved in Lemma \ref{first version of AdT Lemma} that the normalized Chern character is rationally surjective and we know by the smoothing theorem that $\Theta$ is a rational isomorphism.

\begin{proof} Take the diagram from Lemma \ref{representation by immersed Hatcher} and add the arrow $\Theta$:
\[
\xymatrix{
G(L,\d_0L)\ar[d]_{\Sig_{\tilde\ll}}
%\ar[rr]_{top\,E_L^n(-,\eta)}
\ar[drr]^{top\,E_+^n(M,\tilde\ll,-)}
\ar[rrr]^{D\circ(-1)^n2\normch} 
&&  
%\pi_0\std{L}{\d_0}(E)\ar[d]^{D(\tilde\ll)_\ast} %\ar[r]_(.4){\Theta}
%\ar@/^2pc/[rr]_{\t^\IK}
&
% H_{q-4\bullet}(E,E_{\d_1})\ar[d]^{D(\tilde\ll)_\ast}\ar[r]_\cong
 H_{q-4\bullet}(L,\d_1 L)%\ar[d]^{\ll_\ast}
 \ar[d]^{\Sig_{\tilde\ll}}
 \\
\ovsdgo{B}{\d_0}(M)\ar[rr]_{\ov{AdT}}
\ar@/_2pc/[rrr]_{(-1)^n2D\normch}
&&  
\pi_0\std{B}{\d_0}(M)\ar[r]_(.4){\Theta}
%\ar@/_2pc/[rr]_{\t^\IK}
&
 H_{q-4\bullet}(M,M_{\d_1B})%\ar[r]_{p_\ast} 
 &
 %H_{q-4\bullet}(B,\d_1B)
	}%end xymatrix
\]
The outside curved square commutes by Theorem \ref{first main theorem}. The map $\Sig_{\tilde\ll}$ can be chosen to hit any element of $\ovsdgo{B}{\d_0}(M)$ by the previous lemma. Therefore, the curved triangle at the bottom commutes. This implies the lemma since the maps factor uniquely through $\ovsdgo{B}{\d_0}(M)$.
\end{proof}

% subsection

%\newpage
%-----------------------------------------------------------------------------------
%            sub section {stratified deformation argument}
%-----------------------------------------------------------------------------------

\subsection{Stratified deformation lemma}\label{subsecA32}

It remains to prove the following lemma which was used to show that each Arc de Triomphe construction can be deformed into an immersed Hatcher construction.

\begin{lem}[Stratified Deformation Lemma]\label{stratified deformation lemma}
If the fiber dimension of $M$ is $\ge q+2$, then any element of $\sdgo{B}{\d_0}(M)$ is represented by a stratified subset $(\Sig,\psi)$ of $M$ with the property that the components of $\Sig_-$ are contained in disjoint contractible subsets of $\Sig$.
\end{lem}

\begin{proof} This is the same proof which appears in \cite{I:FF} on page 446-447 with five figures and in \cite{I:ComplexTorsion} on page 73 with one figure. We repeat the argument and pictures here since the statements are not the same, only analogous.

To clarify the statement of this lemma we point out that the mushroom ($t=0$ in Example \ref{eg:mushroom}) is already in the desired form since $\Sig_-\cong S^{k-1}\times I$ is contained in the contractible subset $\Sig_-\cup T_+\cong D^k$ of $\Sig$. Thus the contractible set can contain parts of $\Sig_+$.

The dimension hypothesis implies that all deformations of $\Sig$ in $M$ can be made into isotopies of smooth embeddings over $B$ by transversality. So, we will not concern ourselves with that point. Also, by Remark \ref{rem: vertical vector field}, there is a nowhere zero vertical vector field along any stratified subset $\Sig\subseteq M$ which points from $\Sig_-$ to $\Sig_+$ along $\Sig_0$. As in the proof of Proposition \ref{sd is a group} we can use this to find a ribbon $R_-\cong \Sig_-\times I$ containing $\Sig_-$ in $M$.

Suppose that $\d_1B$ is empty. {Then we will deform any $(\Sig,\psi)$ into the desired shape} (so that the union of $\Sig_-$ and a portion of $\Sig_+$ is a contractible subset of $\Sig$.) When $\d_1B$ is nonempty, we double $B$ along $\d_1B$ and double $M$ along $M_{\d_1B}$ and similarly for $(\Sig,\psi)$. Then do the deformation $\ZZ/2$ equivariantly. The fixed point sets of the $\ZZ/2$ action on the new $B$ and new $M$ are the original $\d_1B$ and $M_{\d_1B}$.

First choose an equivariant triangulation of $\Sig_-$ so that the fixed point set is a subcomplex and so that each simplex maps monomorphically into $B$. Then we will cut apart the set $\Sig_-$ by deleting a tubular neighborhood of each interior simplex $\Delta^m$ starting with the lowest dimension $m=0$. Let $w$ be a vertex in the interior of $\Sig_-$. Near $w$ we embed the ribbon $\Sig_-\times I$ and we will perform the deformation completely inside of this ribbon. 

($m=0$) The desired stratified deformation is given in Example \ref{eg:mushroom} but with the labels $\Sig_-,\Sig_+$ reversed and with $k=q$. In words, we create a $q$-lens above the point $w$ (above means in the direction of the vector field $v$ of Remark \ref{rem: vertical vector field}) together with the coefficient map sending the entire $q$-lens to $\psi(w)\in G/O$. This is the $t=\frac12$ part of Example \ref{eg:mushroom}. Then we attach the mushroom and cancel the portion of $\Sig_-$ around $w$ as in the $t=0$ picture of the example. Remembering that we have reversed $\Sig_-,\Sig_+$ we see that the new $\Sig_-$ is the disjoint union of $D^q$ (the ``top'' of the mushroom) and the old $\Sig_-$ with a disk shaped hole cut out around $w$. In other words, we have \emph{removed} a neighborhood of $w$ from the old $\Sig_-$.

($m=1$) Next, take a 1-simplex in $\Sig_-$. Since we have attached mushrooms on the two endpoints, the picture of this 1-simplex is as follows. Since $\Sig$ is $q$-dimensional, we have the product with $D^{q-1}$ in a small neighborhood of the 1-simplex where the endpoints stay inside the stems of the mushrooms planted on the endpoints.

\begin{figure}[htbp]
\begin{center}
%
%\vs5
{
\setlength{\unitlength}{.7in}
%\centerline
{\mbox{
\begin{picture}(5,1)
\put(0,0){ % begin A
\thinlines
      \qbezier(1,.75)(1.5,.75)(1.5,.65)
      \qbezier(1.35,.5)(1.5,.5)(1.5,.65)
      \qbezier(1.35,.5)(1.2,.5)(1.2,.4)      
    \thicklines
      \qbezier(1.4,.3)(1.8,.25)(2,.28)
      \qbezier(1.4,.3)(1.2,.3)(1.2,.4)      
%      \put(2,0)
      \put(1.7,.35){$\Sigma_-$}
%      \put(1.2,.85){$T_+$}
%     \put(1,0){ \line(0,1){1}}
     %
%   \put(.85,-.4){$t=0$}
   } %end A
    \thicklines
   \put(2,.27){\line(1,0){1}}
   \put(3,0){ % begin C
      \thinlines
      \put(2,0){
      \qbezier(-1,.75)(-1.5,.75)(-1.5,.65)
      \qbezier(-1.35,.5)(-1.5,.5)(-1.5,.65)
      \qbezier(-1.35,.5)(-1.2,.5)(-1.2,.4)      
     \thicklines
       \qbezier(-1.4,.3)(-1.8,.25)(-2,.28)
     \qbezier(-1.4,.3)(-1.2,.3)(-1.2,.4)      
      }
      \put(1.5,.45){$\times\ D^{q-1}$}
%      \put(1.7,.35){$\Sigma_+$}
 %     \put(1.2,.85){$T_+$}
%     \put(1,0){ \line(0,1){1}}
     %
%   \put(.85,-.4){$t=0$}
   } %end C
\end{picture}}
}}
%\vs5
%\caption{xxx}
\end{center}
\end{figure}

We focus attention to a small neighborhood of the 1-simplex in $\Sig_-$ (ignoring the $\Sig_-$ tops of the old mushrooms). Then, on the two $\Sig_+$ segments ($\times D^{q-1}$) we plant two new mushrooms (with $T_+\subseteq \Sig_+$ tops) and perform the deformation in Figure \ref{fig: deleting an edge}. When one of the endpoints of the 1-simplex lies on the boundary of the original set $\Sig_-$, we will not have a mushroom and the figure above is not quite accurate in that case. However, we will still have a $\Sig_+$ segment which meets the boundary of $\Sig_-$ along $\Sig_0$ and we can still perform this deformation (plant two mushrooms) and Figure \ref{fig: deleting an edge} will be accurate in this case.

When we plant the two new mushrooms, a new $\Sig_-$ component $S$ diffeomorphic to $S^{q-1}\times I\times S^0$ is created. This is contained in $S\cup T_+\cong D^q\times S^0$. When we do the deformation indicated, we attach a solid $q$-handle to this to form a contractible subset of the new $\Sig$ containing the new component of $\Sig_-$ (which is now homotopy equivalent to a wedge of two $q-1$ spheres, but with each sphere filled in with a $q$-disk $T_+ \subset \Sig_+$) we also extend the coefficient map by using the values of $\psi$ on the old 1-simplex in $\Sig_-$ at the bottom of the figure. This old 1-simplex is now at the bottom of a 1-lens and this can be cancelled by the deformation obtained by rotating the lens since the top of lens and the bottom of the lens have matching coefficient maps. The result is that the set $\Sig_-$ is changed by the deletion of the 1-simplex. In more standard language, this last step performs surgery on a circle $S^1$ embedded in $\Sig$ so that half the circle is in $\Sig_+$ and half is in $\Sig_-$. This second half is the 1-simplex which has been ``eliminated.''

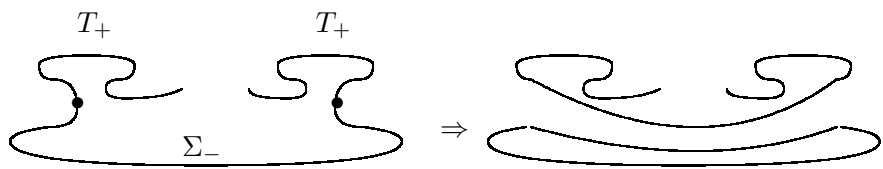
\begin{figure}[htbp]
\begin{center}
%
%\vs5
{
\setlength{\unitlength}{.5in}
%\centerline
{\mbox{
\begin{picture}(9,1.5)
\put(0,0){ % begin A
 % begin B
 \put(0.9,1.3){$T_+$}
  \thinlines
      \put(2,0.3){
      \qbezier(-1,.75)(-1.5,.75)(-1.5,.65)
      \qbezier(-1.35,.5)(-1.5,.5)(-1.5,.65)
%      \qbezier(-1.35,.5)(-1.2,.5)(-1.2,.4)  
       \qbezier(-1.35,.5)(-1.1,.5)(-1.1,.2)
    \put(-1.17,.18){$\bullet$} 
  \qbezier(-1.39,0)(-1.1,0)(-1.1,.2)      
%     \qbezier(-1.4,.3)(-1.2,.3)(-1.2,.4) 
      }
      \put(2,0){       \put(0,0){$\Sig_-$}
      \qbezier(-1.4,.3)(-1.8,.25)(-1.8,.15)}
 \put(0,.3){     \qbezier(1,.75)(1.5,.75)(1.5,.65)
      \qbezier(1.35,.5)(1.5,.5)(1.5,.65)
      \qbezier(1.35,.5)(1.2,.5)(1.2,.4)      
      \qbezier(1.4,.3)(1.8,.3)(2,.4)
      \qbezier(1.4,.3)(1.2,.3)(1.2,.4)      
       \qbezier(1,.75)(1.5,.75)(1.5,.65)   }
       \qbezier(.2,.15)(.2,-.1)(2.25,-.1)  
       \qbezier(4.3,.15)(4.3,-.1)(2.25,-.1)  
   %end B
%    \thicklines
   \put(2.5,0){ % begin C
\thinlines
 \put(0.9,1.3){$T_+$}
      \put(2,.3){
      \qbezier(-1,.75)(-1.5,.75)(-1.5,.65)
      \qbezier(-1.35,.5)(-1.5,.5)(-1.5,.65)
      \qbezier(-1.35,.5)(-1.2,.5)(-1.2,.4)      
       \qbezier(-1.4,.3)(-1.8,.3)(-1.8,.4)
    \qbezier(-1.4,.3)(-1.2,.3)(-1.2,.4)   
      }
\put(0,.3){      \qbezier(1.35,.5)(1.5,.5)(1.5,.65)
%      \qbezier(1.35,.5)(1.2,.5)(1.2,.4)  
% \qbezier(1.35,.5)(0,-.5)(-1.85,.5)    %new
  \qbezier(1.35,.5)(1.1,.5)(1.1,.2)
    \put(1.05,.18){$\bullet$} 
  \qbezier(1.39,0)(1.1,0)(1.1,.2)
% \qbezier(1.35,-0)(0,-.5)(-1.85,-0)    %new
       \qbezier(1,.75)(1.5,.75)(1.5,.65)     
      }
 %     \qbezier(1.4,.3)(1.2,.3)(1.2,.4)      
      \qbezier(1.4,.3)(1.8,.25)(1.8,.15)
  } %end C
  } %end A
  \put(4.6,0.2){ $\then$}
  \put(5,0){ % begin X
  \thinlines
      \put(2,0.3){
      \qbezier(-1,.75)(-1.5,.75)(-1.5,.65)
      \qbezier(-1.35,.5)(-1.5,.5)(-1.5,.65)
%      \qbezier(-1.35,.5)(-1.2,.5)(-1.2,.4)      
%     \qbezier(-1.4,.3)(-1.2,.3)(-1.2,.4) 
      }
      \put(2,0){       \qbezier(-1.4,.3)(-1.8,.25)(-1.8,.15)}
 \put(0,.3){     \qbezier(1,.75)(1.5,.75)(1.5,.65)
      \qbezier(1.35,.5)(1.5,.5)(1.5,.65)
      \qbezier(1.35,.5)(1.2,.5)(1.2,.4)      
      \qbezier(1.4,.3)(1.8,.3)(2,.4)
      \qbezier(1.4,.3)(1.2,.3)(1.2,.4)      
       \qbezier(1,.75)(1.5,.75)(1.5,.65)   }
       \qbezier(.2,.15)(.2,-.1)(2.25,-.1)  
       \qbezier(4.3,.15)(4.3,-.1)(2.25,-.1)  
   %end X
%    \thicklines
   \put(2.5,0){ % begin Z
\thinlines
      \put(2,.3){
      \qbezier(-1,.75)(-1.5,.75)(-1.5,.65)
      \qbezier(-1.35,.5)(-1.5,.5)(-1.5,.65)
      \qbezier(-1.35,.5)(-1.2,.5)(-1.2,.4)      
       \qbezier(-1.4,.3)(-1.8,.3)(-1.8,.4)
     \qbezier(-1.4,.3)(-1.2,.3)(-1.2,.4)      
      }
\put(0,.3){      \qbezier(1.35,.5)(1.5,.5)(1.5,.65)
%      \qbezier(1.35,.5)(1.2,.5)(1.2,.4)  
 \qbezier(1.35,.5)(0,-.5)(-1.85,.5)    %new
 \qbezier(1.35,-0)(0,-.5)(-1.85,-0)    %new
       \qbezier(1,.75)(1.5,.75)(1.5,.65)     
      }
 %     \qbezier(1.4,.3)(1.2,.3)(1.2,.4)      
      \qbezier(1.4,.3)(1.8,.25)(1.8,.15)
  } %end Z
  }
\end{picture}}
}}
%\vs5
\caption{Plant two new mushrooms and cancel the two points indicated with spots.}
\label{fig: deleting an edge}
\end{center}
\end{figure}

($m\ge2$) {Suppose by induction that the $m-1$ skeleton of $\Sig_-$ has been removed where $m\ge2$. Let $D^m$ be what remains of one of the original $m$-simplices of $\Sig_-$. Then $D^m$ has boundary $S^{m-1}\subseteq\Sig_0$. Part of this boundary comes from the original boundary of $\Sig_-$ and the other part comes from the inductive procedure. There are remnants of mushrooms from previous steps in the construction and we need to avoid them and use only those structures which exist in all parts of the boundary of $D^m\subset\Sig_-$.

Since $\Sig$ is $q$ dimensional, this disk sits in $D^m\times D^{q-m}\subset \Sig_-$. The next step in the deformation is given by planting the product of $S^{m-1}$ with a mushroom of dimension $q-m+1$. The picture is the same as Figure \ref{fig: deleting an edge}. So, we do not redraw it. However, we give a new interpretation of the same figure.}

Take the left hand figure in Figure \ref{fig: deleting an edge}. This is a planar figure which is now being spun around the middle vertical axis over all $\th\in S^{m-1}$. The tops of the mushrooms, which are given locally by $t=0$ in Example \ref{eg:mushroom}, become diffeomorphic to $S^{m-1}\times I$. For all $z\in D^{q-m}$, we replace these mushrooms with the $t=||z||$ picture from Example \ref{eg:mushroom} and spin around $S^{m-1}$. This gives a stratified set over $D^m\times D^{q-m}$ which contains a new components $S\subset \Sig_-$ diffeomorphic to $S^{m-1}\times I\times S^{q-m}$. However, with the tops of the mushrooms we get $S\cup T_+\cong S^{m-1}\times D^{q-m+1}$. The deformation (passing from left to right) in Figure \ref{fig: deleting an edge} is to be carried out only for $z$ close to the origin in $D^{q-m}$ otherwise the points indicated with spots are not in the picture and, again, we use the value of the coefficient map on the $m$-simplex in $\Sig_-$ to extend the value of $\psi$ to the top of the new $m$-lens that we have formed. This deformation performs surgery on a $m-1$ sphere in $\Sig$ which lies in $\Sig_0$ on the boundary of $S\cup T_+$. This changes $S\cup T_+$ into a $q$-disk. So, the new component of $\Sig_-$ is contained in a contractible subset of $\Sig$.

On the right hand side of Figure \ref{fig: deleting an edge} we have an $m$-lens which can be eliminated by Example \ref{eg:k-lens} since the value of $\psi$ on top and bottom match by construction. This performs surgery on an $m$-sphere in $\Sig$ which meets $\Sig_-$ is an $m$-disk which is the remains of the $m$-simplex which we are trying to eliminate. This deformation therefore completes the induction and proves the lemma.
\end{proof}

This completes the proof of all the theorems in this paper.

%%%%%%%%%%%%%%%%%%%%%%%%%%
%
%                END OF Part C
%
%%%%%%%%%%%%%%%%%%%%%%%%%%

%\appendix

 %%%%%%%%

%\newpage
 %%%%%%%%%%%%%%%%%%%%%%%%%%%%%%%%%%%%
 %
 %					bibliography 
 %
 %%%%%%%%%%%%%%%%%%%%%%%%%%%%%%%%%%%%

\bibliographystyle{amsplain}

 %%%%%%%%%%%%%%%%%%%%%%%%%%%%%%%%%%%%
 %
 %					end of the document 
 %
 %%%%%%%%%%%%%%%%%%%%%%%%%%%%%%%%%%%%

\end{document}